\documentclass[12pt,a4paper,psamsfonts]{amsart}
\usepackage{fancyhdr}
\usepackage{appendix}
\usepackage{amssymb,amscd,amsxtra,calc}
\usepackage{mathrsfs}
\usepackage{cmmib57}
\usepackage{multirow}
\usepackage[all]{xy}
\usepackage[colorlinks,linkcolor=blue,anchorcolor=blue,citecolor=green]{hyperref}
\theoremstyle{plain}
    \newtheorem{thm}{Theorem}[section]

    \newtheorem{example}[thm]{Example}
    \newtheorem{lemma}[thm]{Lemma}
    \newtheorem{proposition}[thm]{Proposition}
    \newtheorem{question}[thm]{Question}
    \newtheorem{theorem}[thm]{Theorem}

\theoremstyle{definition}
    \newtheorem{definition}[thm]{Definition}

    \newtheorem{remark}[thm]{Remark}
\theoremstyle{remark}

    \newtheorem{setup}[thm]{}

\newcommand{\Q}{\mathbb{Q}}

\newcommand{\R}{\mathbb{R}}

\newcommand{\Z}{\mathbb{Z}}

\newcommand{\Prep}{\operatorname{Prep}}

\newcommand{\alb}{\operatorname{alb}}
\newcommand{\Amp}{\operatorname{Amp}}

\newcommand{\id}{\operatorname{id}}

\newcommand{\Nef}{\operatorname{Nef}}
\newcommand{\NS}{\operatorname{NS}}

\newcommand{\PE}{\operatorname{PE}}

\newcommand{\Sing}{\operatorname{Sing}}
\newcommand{\SL}{\operatorname{SL}}
\newcommand{\Supp}{\operatorname{Supp}}

\newcommand{\Per}{\operatorname{Per}}

\newcommand{\N}{\operatorname{N}}

\newcommand{\Alb}{\operatorname{Alb}}

\begin{document}

\title[Amplified endomorphisms of normal projective varieties]
{Building blocks of amplified endomorphisms of normal projective varieties
}

\author{Sheng Meng}

\address
{
\textsc{Department of Mathematics} \endgraf
\textsc{National University of Singapore,
Singapore 119076, Republic of Singapore
}}
\email{ms@u.nus.edu}

\begin{abstract}
Let $X$ be a normal projective variety.
A surjective endomorphism $f:X\to X$ is int-amplified if $f^\ast L - L =H$  for some ample Cartier divisors $L$ and $H$.
This is a generalization of the so-called polarized endomorphism which requires that $f^*H\sim qH$ for some ample Cartier divisor $H$ and $q>1$.
We show that this generalization keeps all nice properties of the polarized case in terms of the singularity, canonical divisor, and equivariant minimal model program.
\end{abstract}

\subjclass[2010]{
14E30,   
32H50, 
08A35.  
}

\keywords{amplified endomorphism, iteration, equivariant MMP, $Q$-abelian variety, Albanese morphism, Albanese map, MRC fibration}

\maketitle
\tableofcontents

\section{Introduction}
We work over an algebraically closed field $k$ which has characteristic zero.

Let $f$ be a surjective endomorphism of a projective variety $X$. 
We say that $f$ is {\it polarized} if $f^*L\sim qL$ for some ample Cartier divisor $L$ and integer $q>1$.
We say
that $f$ is {\it int-amplified} if $H:=f^*L-L$ is ample for some ample Cartier divisor $L$.

We refer to \cite[\S2]{MZ} for the definitions and properties of the numerical equivalence ($\equiv$) of $\R$-Cartier divisors and the weak numerical equivalence ($\equiv_w$) of $r$-cycles with $\R$-coefficients.
Denote by $\N^1(X):=\NS(X) \otimes_{\Z} \mathbb{R}$ where $\NS(X)$ is the N\'eron-Severi group of $X$. 
Denote by $\N_r(X)$ the quotient vector space of $r$-cycles modulo the weak numerical equivalence.
Any surjective endomorphism $f$, via pullback, induces invertible linear maps on $\N^1(X)$ and $\N_r(X)$, denoted by $f^*|_{\N^1(X)}$ and $f^*|_{\N_r(X)}$.
We first give the following criterion for int-amplified endomorphisms. From this, one can easily see that it is very natural to define and study such kind of endomorphisms. We refer to \cite[Proposition 2.9]{MZ} for a criterion for polarized endomorphism. 
 
\begin{theorem}\label{main-thm-cri} Let $f:X\to X$ be a surjective endomorphism of a projective variety $X$.
Then the following are equivalent.
\begin{itemize}
\item[(1)] The endomorphism $f$ is int-amplified.
\item[(2)] All the eigenvalues of $\varphi:=f^*|_{\N^1(X)}$ are of modulus greater than $1$.
\item[(3)] There exists some big $\R$-Cartier divisor $B$ such that $f^*B-B$ is big.
\item[(4)] If $C$ is a $\varphi$-invariant convex cone in $\N^1(X)$,
then $\emptyset\neq(\varphi-\id_{\N^1(X)})^{-1}(C)\subseteq C$.
\end{itemize}
\end{theorem}

\begin{remark}\label{main-rmk-1}
The approach towards Theorem \ref{main-thm-cri} is purely cone theoretical.
Therefore, it also applies to the action $f^*|_{\N_{n-1}(X)}$; see Theorem \ref{thm-cri2} for the precise argument.
\end{remark}

One advantage of studying int-amplified endomorphisms is that, with it, the category of polarized endomorphisms is largely extended to include taking the product.
Note that, in general, an int-amplified endomorphism may not split as a product of polarized endomorphisms; see Example \ref{exa-intamp}.
For the compositions of maps,
X. Yuan and S. Zhang asked the following question.
Unfortunately, it has a negative answer; see Example \ref{exa-comp}. 
However, we are able to show in Theorem \ref{main-thm-comp} that the composition of sufficient iterations of int-amplified endomorphisms is still int-amplified.
The proof essentially uses Theorem \ref{main-thm-cri}.

\begin{question}(cf.~\cite[Question 4.15]{YZ})\label{que-yz}
Let $f$ and $g$ be polarized endomorphisms of a projective variety $X$ such that $\Prep(f)=\Prep(g)$ where $\Prep$ is the set of preperiodic points.
Is $f\circ g$ polarized ?
\end{question}

\begin{theorem}\label{main-thm-comp}
Let $f$ and $g$ be surjective endomorphisms of a projective variety $X$.
Suppose $f$ is int-amplified.
Then $f^i\circ g$ and $g\circ f^i$ are int-amplified when $i\gg 1$.
\end{theorem}

In the rest of this paper, we focus on showing that int-amplified endomorphisms keep all the nice properties of polarized endomorphisms concerning the canonical divisor, singularity and equivariant minimal model program (MMP).
The main technique required for this generalization is in Section \ref{sec-prop} which applies the intersection theory.
We refer to \cite{MZ} and \cite{CMZ} for the details about the polarized case.

Let $f:X\to X$ be a surjective endomorphism of a normal projective variety $X$.
When $f$ is polarized and $X$ is smooth, Boucksom, de Fernex and Favre \cite[Theorem C]{BFF} showed that $-K_X$ is pseudo-effective.
Cascini, Zhang and the author \cite[Theorem 1.1 and Remark 3.2]{CMZ} used a different method to
show further that $-K_X$ is weakly numerically equivalent to an effective Weil $\Q$-divisor without the assumption of $X$ being smooth.
Now applying Theorem \ref{main-thm-cri} to the ramification divisor formula $K_X=f^*K_X+R_f$ where $R_f$ is the ramification divisor for $f$,
this result can be easily generalized to the int-amplified case.

\begin{theorem}\label{main-thm-k} Let $X$ be a normal projective variety admitting an int-amplified endomorphism.
Then $-K_X$ is weakly numerically equivalent to some effective Weil $\Q$-divisor.
If $X$ is further assumed to be $\Q$-Gorenstein, then $-K_X$ is numerically equivalent to some effective $\Q$-Cartier divisor.
\end{theorem}

We refer to \cite[Chapters 2 and 5]{KM} for the definitions and the properties of log canonical (lc), Kawamata log terminal (klt), canonical and terminal singularities.
Let $f:X\to X$ be a non-isomorphic surjective endomorphism of a normal projective variety $X$.
Wahl \cite[Theorem 2.8]{Wa} showed that
$X$ has at worst lc singularities when $\dim(X)=2$.
Broustet and H\"oring \cite[Corollary 1.5]{BH} generalized this result to the higher dimensional case with additional assumptions that $f$ is polarized and $X$ is $\Q$-Gorenstein.
We generalize their result to the int-amplified case in the following.

\begin{theorem}\label{main-thm-lc} Let $X$ be a $\Q$-Gorenstein normal projective variety admitting an int-amplified endomorphism.
Then $X$ has at worst lc singularities.
\end{theorem}

Let $f:X\to X$ be a surjective endomorphism of a normal projective variety $X$. 
We consider typical $f$-equivariant morphisms; see
\cite[\S4]{Na10} or Definition \ref{def-smrc} for the special maximal rationally connected (MRC) fibration, and see
also \cite[Remark 9.5.25]{FGI}, \cite[Chapter II.3]{Lang} and \cite[\S5]{CMZ} for the Albanese morphism and the Albanese map (cf.~Section \ref{sec-alb}).

The result below is a generalization of \cite[Proposition 1.6]{MZ}.
A normal projective variety $X$ is said to be $Q$-abelian if there is a finite surjective morphism $\pi:A\to X$ \'etale in codimension $1$ with $A$ being an abelian variety.
\begin{theorem}\label{main-thm-mrc} Let $f:X\to X$ be an int-amplified endomorphism of a normal projective variety $X$.
Then there is a special MRC fibration $\pi:X\dashrightarrow Y$ in the sense of Nakayama \cite{Na10}
(which is the identity map when $X$ is non-uniruled)
together with a (well-defined) surjective endomorphism $g$ of $Y$, such that the following are true.
\begin{itemize}
\item[(1)]
$g\circ \pi=\pi\circ f$; $g$ is int-amplified.
\item[(2)]
$Y$ is a $Q$-abelian variety with only canonical singularities. 
\item[(3)]
Let $\bar{\Gamma}_{X/Y}$ be the normalization of the graph of $\pi$. Then the induced morphism $\bar{\Gamma}_{X/Y}\to Y$ is equi-dimensional
with each fibre (irreducible) rationally connected.
\item[(4)]
If $X$ has only klt singularities, then $\pi$ is a morphism.
\end{itemize}
\end{theorem}

The following result answers Krieger - Reschke \cite[Question 1.10]{Kr} when $f$ is int-amplified. For the polarized case (especially in arbitrary characteristic), see \cite[Corollary  1.4]{MZ} and \cite[Theorem 1.2]{CMZ}.
\begin{theorem}\label{main-thm-alb} Let $f:X\to X$ be an int-amplified endomorphism of a normal projective variety $X$. Then we have the following.
\begin{itemize}
\item[(1)]
The Albanese morphism $\alb_X:X\to \Alb(X)$ is surjective with $(\alb_X)_\ast \mathcal{O}_X=\mathcal{O}_{\Alb(X)}$ and all the fibres of $\alb_X$ are irreducible and equi-dimensional. The induced morphism $g:\Alb(X)\to \Alb(X)$ is int-amplified.
\item[(2)]
The Albanese map  $\mathfrak{alb}_X:X\dashrightarrow \mathfrak{Alb}(X)$ is dominant and the induced morphism $h:\mathfrak{Alb}(X)\to \mathfrak{Alb}(X)$ is int-amplified.
\end{itemize}
\end{theorem}

By Theorems \ref{main-thm-mrc} and \ref{thm-kx-pe-qa}, we have the following result.
\begin{theorem}\label{main-thm-qa}
Let $f:X\to X$ be an int-amplified endomorphism of a normal projective variety $X$.
Suppose either $X$ is klt and $K_X$ is pseudo-effective or $X$ is non-uniruled.
Then $X$ is $Q$-abelian.
\end{theorem}

Finally, we generalize the result of equivariant MMP \cite[Theorem 1.8]{MZ} to the int-amplified case.
Note that we need the key observation Lemma \ref{lem-diag-fano} to show (3) below.

\begin{theorem}\label{main-thm-mmp}
Let $f:X\to X$ be an int-amplified endomorphism of a $\mathbb{Q}$-factorial klt projective variety $X$.
Then, replacing $f$ by a positive power, there exist a $Q$-abelian variety $Y$, a morphism $X\to Y$, and an $f$-equivariant relative MMP over $Y$ $$X=X_1\dashrightarrow \cdots \dashrightarrow X_i \dashrightarrow \cdots \dashrightarrow X_r=Y$$ (i.e. $f=f_1$ descends to $f_i$ on each $X_i$), with every $X_i \dashrightarrow X_{i+1}$ a divisorial contraction, a flip or a Fano contraction, of a $K_{X_i}$-negative extremal ray, such that we have:
\begin{itemize}
\item[(1)]
If $K_X$ is pseudo-effective, then $X=Y$ and it is $Q$-abelian.
\item[(2)]
If $K_X$ is not pseudo-effective, then for each $i$, $f_i$ is int-amplified and $X_i\to Y$ is an equi-dimensional morphism with every fibre irreducible. All the fibres are rationally connected if the base field is uncountable.
The $X_{r-1}\to X_r = Y$ is a Fano contraction.
\item[(3)]
$f^*|_{\N^1(X)}$ is diagonalizable over $\mathbb{C}$ if and only if so  is $f_{r}^\ast|_{\N^1(Y)}$.
\end{itemize}
\end{theorem}

The following theorem is an application of Theorem \ref{main-thm-mmp} and generalizes \cite[Theorem 1.10(1)]{MZ}.
\begin{theorem}\label{main-thm-rc-diag} Let $f:X\to X$ be an int-amplified endomorphism of a smooth rationally connected projective variety $X$.
Then there exists some $s>0$, such that $(f^s)^*|_{\N^1(X)}$ is diagonalizable over $\Q$ with all the eigenvalues being positive integers greater than $1$.
In particular, $f^*|_{\N^1(X)}$ is diagonalizable over $\mathbb{C}$.
\end{theorem}

When $f:X\to X$ is a polarized endomorphism of a projective variety $X$, the action $f^*|_{\N^1(X)}$ is always diagonalizable over $\mathbb{C}$ and all the eigenvalues are of the same modulus (cf.~\cite[Proposition 2.9]{MZ}).
However, Theorem \ref{main-thm-rc-diag} fails without the assumption of rational connectedness due to Example \ref{exa-fak} given by Najmuddin Fakhruddin.

\begin{remark}[Differences with early papers]
In the papers of \cite{MZ} and \cite{CMZ}, for a polarized $f:X\to X$ with $f^*H\sim qH$ where $q>1$ and $H$ is ample, the nice eigenvector $H$ of $f^*$ is frequently used. For example, by taking top self-intersection of $H$ and the projection formula, one can easily see that $\deg f=H^{\dim(X)}$.
However, for the int-amplified case, there is no such simple way. A rough bound is given in Lemma \ref{lem-deg>1} and it is precisely characterized in the proof of Lemma \ref{lem-diag-fano}.
On the other hand, the pullback action of a polarized $f$ on $\N_r(X)$ is  clearly characterized (cf.~\cite[Lemma 2.4]{Zh-comp} and \cite[Theorem 1.1]{Zhsw}).
For the int-amplified case, we are only able to give a ``limit'' version in Lemma \ref{lem-lim-int}.
Due to these difficulties, all the generalizations as shown in the previous main theorems are required to adjust the old proofs in \cite{MZ} and \cite{CMZ} accordingly based on the new methods in Section \ref{sec-prop}.
Finally, we highlight that the int-amplified criteria in Theorem \ref{main-thm-cri} by the cone analysis are the keys to making all the subsequent methods and results possible.
\end{remark}

The proofs of Theorems \ref{main-thm-cri}, \ref{main-thm-comp}, \ref{main-thm-k} and \ref{main-thm-lc} are in Section \ref{sec-prop}.
The proof of Theorem \ref{main-thm-mrc} is in Section \ref{sec-mrc}.
The proof of Theorem \ref{main-thm-alb} is in Section \ref{sec-alb}.
The proofs of Theorems \ref{main-thm-mmp} and \ref{main-thm-rc-diag} are in Section \ref{sec-proof}.

\par \vskip 1pc
{\bf Acknowledgement.}
The author would like to thank Professor De-Qi Zhang for many inspiring discussions, Professor Najmuddin Fakhruddin for providing Example \ref{exa-fak}, and the anonymous colleague for the suggestion of Example \ref{exa-amplified}.
He thanks the referee for very careful reading and many useful suggestions to revise this paper.
He also thanks Max Planck Institute for Mathematics for providing an impressive acadamic environment.
The author is supported by a Research Assistantship of the National University of Singapore.

\section{Preliminaries}

\begin{setup}{\bf Notation and terminology}\label{nat}.

Let $X$ be a projective variety.
We use Cartier divisor $H$ (always meaning integral, unless otherwise indicated)
and its corresponding invertible sheaf $\mathcal{O}(H)$ interchangeably.

Let $f:X\to X$ be a surjective endomorphism.
A subset $Z \subseteq X$ is said to be {\it $f$-invariant} (resp. {\it $f^{-1}$-invariant}) if
$f(Z) = Z$ (resp. $f^{-1}(Z) = Z$).
We say that $Z \subseteq X$ is {\it $f$-periodic} (resp. {\it $f^{-1}$-periodic}) if
$f^s(Z) = Z$ (resp. $f^{-s}(Z) = Z$) for some $s > 0$.

Denote by $\Per(f)$ the set of all $f$-periodic closed points.

Let $n:=\dim(X)$.
We can regard $\N^1(X):=\NS(X)\otimes_{\Z}\R$ as the space of numerically equivalent classes of $\R$-Cartier divisors.
Denote by $\N_r(X)$ the space of weakly numerically equivalent classes of $r$-cycles with $\R$-coefficients (cf.~\cite[Definition 2.2]{MZ}).
When $X$ is normal, we also call $\N_{n-1}(X)$ the space of weakly numerically equivalent classes of Weil $\R$-divisors.
In this case, $\N^1(X)$ can be regarded as a subspace of $\N_{n-1}(X)$.
We recall the following $f^*$-invariant cones:
\begin{itemize}
\item $\Amp(X)$: the cone of ample classes in $\N^1(X)$,
\item $\Nef(X)$: the cone of nef classes in $\N^1(X)$,
\item $\PE^1(X)$: the cone of pseudo-effective classes in $\N^1(X)$, and
\item $\PE_{n-1}(X)$: the cone of pseudo-effective classes in $\N_{n-1}(X)$.
\end{itemize}
We refer to \cite[\S2]{MZ} for more information.

Given a finite surjective morphism $\pi:X\to Y$ of two normal projective varieties.
There is a ramification divisor formula 
$$K_X=\pi^*K_Y+R_{\pi}$$
where $R_{\pi}$ is the ramification divisor of $\pi$ which is an integral effective Weil divisor of $X$.
We say that $\pi$ is {\it quasi-\'etale} if $\pi$ is \'etale in codimension 1, i.e., $R_{\pi}=0$.
The purity of branch locus tells us that if $\pi$ is quasi-\'etale and $Y$ is smooth, then $\pi$ is \'etale.
\end{setup}

The inspiration for studying int-amplified endomorphisms comes from the so-called amplified endomorphisms which were first defined by Krieger and Reschke (cf.~\cite{Kr}).
Recall that a surjective endomorphism $f$ is {\it amplified} if $f^*L-L=H$ for some Cartier divisor $L$ and ample Cartier divisor $H$. 
Clearly, ``int-amplified'' is ``amplified'' and Fakhruddin showed the following very motivating result.
\begin{theorem}\label{thm-fak}(cf.~\cite[Theorem 5.1]{Fak}) Let $f:X\to X$ be an amplified endomorphism of a projective variety $X$.
Then the set of $f$-periodic points $\Per(f)$ is Zariski dense in $X$.
\end{theorem}

Fakhruddin's result can be applied to give a rough characterization of projective varieties admitting amplified endomorphisms by the Kodaira dimension. First, we give the following two simple but useful results. 
\begin{lemma}\label{lem-amp-res}
Let $f:X\to X$ be an amplified (resp.~int-amplified) endomorphism of a projective variety $X$.
Let $Z$ be a closed subvariety of $X$ such that $f(Z)=Z$.
Then $f|_Z$ is amplified (resp.~int-amplified).
\end{lemma}
\begin{proof}
Let $i:Z\to X$ be the inclusion map.
Suppose $f^*L-L=H$ for some Cartier divisors $L$ and $H$.
Let $L|_Z:=i^*L$ and $H|_Z:=i^*H$.
Then $(f|_Z)^*(L|_Z)-L|_Z=H|_Z$.
Note that the restriction of an ample Cartier divisor is still ample.
So the lemma is proved.
\end{proof}

\begin{lemma}\label{lem-countable-per} Let $f:X\to X$ be an amplified endomorphism of a projective variety $X$.
Then $\Per(f)$ is countable.
\end{lemma}
\begin{proof} Suppose $\Per(f)$ is uncountable.
Then there exists some $s>0$, such that the set $S$ of all $f^s$-fixed points is infinite.
Let $Z$ be an irreducible component of the closure of $S$ in $X$ with $\dim(Z)>0$.
Then $f^s|_Z=\id_Z$, a contradiction to $f^s|_Z$ being amplified by Lemma \ref{lem-amp-res}.
\end{proof}

\begin{theorem}\label{thm-amp-kod}
Let $f:X\to X$ be an amplified endomorphism of a projective variety $X$.
Then the Kodaira dimension $\kappa(X)\le 0$.
\end{theorem}
\begin{proof}
We may assume $X$ is over the field $k$ which is uncountable by taking the base change.
Suppose $\kappa(X)>0$.
Let $\pi:X\dasharrow Y$ be an Iitaka fibration.
Then $\dim(Y)=\kappa(X)>0$ and $f$ descends to an automorphism $g:Y\to Y$ of finite order by \cite[Theorem A]{NZ09}.
Replacing $f$ by a positive power, we may assume $g=\id_Y$.
Let $U$ be an open dense subset of $X$ such that $\pi$ is well-defined over $U$.
Let $W$ be the graph of $\pi$ and $p_1:W\to X$ and $p_2:W\to Y$ the two projections.
For any closed point $y\in Y$, denote by $X_y:=p_1(p_2^{-1}(y))$ and $U_y:=U\cap X_y$.
Note that $U_{y_1}\cap U_{y_2}=\emptyset$ if $y_1\neq y_2$.
Since $\pi\circ f=\pi$, $f^{-1}(X_y)=X_y$.
Then for some $s_y>0$, $f^{-s_y}(X_y^i)=X_y^i$ for every irreducible component $X_y^i$ of $X_y$, and $f^{s_y}|_{X_y^i}$ is amplified by Lemma \ref{lem-amp-res}.
If $U_y\neq\emptyset$, then $\Per(f)\cap U_y=\Per(f|_{X_y})\cap U_y=\bigcup_i\Per(f^{s_y}|_{X_y^i})\cap U_y\neq \emptyset$ by Theorem \ref{thm-fak}.
Note that $\Per(f)\supseteq\bigcup_{y\in Y} (\Per(f)\cap U_y)$ and there are uncountably many $y\in Y$ such that $U_y\neq\emptyset$.
Then $\Per(f)$ is uncountable, a contradiction to Lemma \ref{lem-countable-per}.
\end{proof}

\begin{remark}
There do exist amplified automorphisms (eg.~automorphisms of positive entropy on abelian surfaces), while the degree of an int-amplified endomorphism is always greater than $1$ (cf.~Lemma \ref{lem-deg>1}).
Unlike the polarized case (cf.~\cite[Corollary 3.12]{MZ}),
it is in general impossible to preserve an amplified automorphism via a birational equivariant lifting (cf.~\cite[Lemma 4.4]{Kr} and \cite[Theorem 1.2]{Re}).
On the other hand, we do not know whether ``amplified'' can be preserved via an equivariant descending (cf.~\cite[Question 1.10]{Kr}).
However, we shall show in the first half of Section \ref{sec-prop} that int-amplified endomorphisms have all these nice properties like the polarized case (cf.~\cite[\S 3]{MZ}).
They are necessary for us to set up the equivariant MMP later in Section \ref{sec-mmp}.
\end{remark}

\begin{remark}\label{rmk-ref}
In general, even after taking an equivariant lifting, an amplified endomorphism may not split into a product of an amplified automorphism and an int-amplified endomorphism; and an int-amplified endomorphism may not split into a product of two polarized endomorphisms; see Section \ref{sec-exa} for the precise argument and examples.
\end{remark}

\section{Properties of int-amplified endomorphisms}\label{sec-prop}

We refer to \cite[Definition 2.6]{MZ} for the notation and symbols involved below.

\begin{lemma}\label{lem-eig>1} Let $\varphi:V\to V$ be an invertible linear map of a positive dimensional real normed vector space $V$.
Let $C$ be a convex cone of $V$ such that $C$ spans $V$ and its closure $\overline{C}$ contains no line.
Assume $\varphi(C)=C$ and $\varphi(\ell)-\ell=h$ for some $\ell$ and $h$ in $C^\circ $ (the interior part of $C$).
Then all the eigenvalues of $\varphi$ are of modulus greater than $1$.
\end{lemma}
\begin{proof}
Note that $\overline{C}^\circ=C^\circ$ since $C$ is a convex cone.
So we may assume $C$ is closed.
Let $\frac{1}{r}$ be the spectral radius of $\varphi^{-1}$.
Note that $\varphi^{\pm}(C)=C$ and $C$ spans $V$ and contains no line.
By a version of the Perron-Frobenius theorem (cf. \cite{Bi}),
$\varphi(v)=rv$ for some nonzero $v\in C$.
Suppose $r\le 1$.
Since $\ell\in C^\circ$ and $v\neq 0$, $\ell-av\in \partial{C}:=C\backslash C^{\circ}$ for some $a>0$.
Then $\varphi(\ell-av)-(\ell-av)=h+a(1-r)v\in C^\circ$.
So $\varphi(\ell-av)\in C^\circ$ and hence $\ell-av\in C^\circ$, a contradiction.
\end{proof}

\begin{proposition}\label{prop-eig-inn}
Let $\varphi:V\to V$ be an invertible linear map of a positive dimensional real normed vector space $V$.
Assume $\varphi(C)=C$ for a convex cone $C\subseteq V$.
Suppose further all the eigenvalues of $\varphi$ are of modulus greater than $1$.
Then $(\varphi-\id_V)^{-1}(C)\subseteq C$.
\end{proposition}

\begin{proof}

Suppose $e:=\varphi(v)-v\in C$.
If $e=0$, then $v=0$ since no eigenvalue of $\varphi$ is $1$.
Next, we assume $e\neq 0$.

For $m\ge 1$, let $E_m$ be the convex cone generated by $\{\varphi^{-1}(e), \cdots, \varphi^{-m}(e)\}$.
Let $E_\infty$ be the convex cone generated by $\{\varphi^{-i}(e)\}_{i\ge 1}$.
Let $E$ be the convex cone generated by $\{\varphi^{-i}(e)\}_{i\in \mathbb{Z}}$.
Then all the above cones are subcones of $C$.
Note that $\varphi^{\pm}(E)=E$ and $\varphi^{-1}(E_\infty)\subseteq E_\infty$.
Let $W$ be the vector space spanned by $E$.
Since $e\neq 0$, $\dim(W)>0$.

We claim that $E_\infty$ spans $W$.
Let $W'$ be the vector space spanned by $E_\infty$.
Then $\varphi^{-1}(W')\subseteq W'$ and hence $\varphi(W')= W'$ since $W'$ is finite dimensional and $\varphi$ is invertible.
In particular, $\varphi^i(e)\in W'$ for any $i\in \mathbb{Z}$ and hence $W\subseteq W'$.
So the claim is proved.

Now we may assume $E_m$ spans $W$ for $m\gg 1$.
This implies $E_m^\circ\subseteq \overline{E}^\circ$.
Therefore, $s_m:=\sum\limits_{i=1}^m \varphi^{-i}(e)\in E_m^{\circ}\subseteq \overline{E}^{\circ}$.
Note that $\lim\limits_{n\to +\infty} \varphi^{-n}(v)= 0$ since  all the eigenvalues of $\varphi$ are of modulus greater than $1$.
Then $v=\lim\limits_{n\to +\infty} v-\varphi^{-n}(v)=\lim\limits_{n\to +\infty}\sum\limits_{i=1}^n \varphi^{-i}(e)=s_m+\lim\limits_{n\to +\infty}\sum\limits_{i={m+1}}^n \varphi^{-i}(e)\in \overline{E}^\circ=E^{\circ}$.
In particular, $v\in E\subseteq C$.
\end{proof}

Now we are able to prove Theorem \ref{main-thm-cri}.
\begin{proof}[Proof of Theorem \ref{main-thm-cri}] 
Let $V:=\N^1(X)$ and $\varphi:=f^*|_{\N^1(X)}$.
It is clear that (1) implies (3), and (4) implies (1) by letting $C=\Amp(X)$.

Suppose all the eigenvalues of $\varphi$ are of modulus greater than $1$.
Then $\varphi-\id_V$ is invertible.
By Proposition \ref{prop-eig-inn}, (2) implies (4).

Suppose $f^*B-B$ is big for some big $\R$-Cartier divisor $B$.
Let $C:=\PE^1(X)$ the cone of all classes of pseudo-effective $\R$-Cartier divisors in $\N^1(X)$.
By applying Lemma \ref{lem-eig>1} to $C$, (3) implies (2).
\end{proof}

Let $X$ be a normal projective variety of dimension $n$ and $D$ a Weil-$\R$ divisor. Recall that $D$ is big if its class $[D]\in \PE_{n-1}(X)$; see \cite[Theorem 3.5]{FKL} for equivalent definitions.
Considering the action $f^*|_{\N_{n-1}(X)}$ and the cone $\PE_{n-1}(X)$, we have similar criteria as follows.
\begin{theorem}\label{thm-cri2} Let $f:X\to X$ be a surjective endomorphism of an $n$-dimensional normal projective variety $X$.
Then the following are equivalent.
\begin{itemize}
\item[(1)] The endomorphism $f$ is int-amplified.
\item[(2)] All the eigenvalues of $\varphi:=f^*|_{\N_{n-1}(X)}$ are of modulus greater than $1$.
\item[(3)] There exists some big Weil $\R$-divisor $B$ such that $f^*B-B$ is a big Weil $\R$-divisor.
\item[(4)] If $C$ is a $\varphi$-invariant convex cone in $\N_{n-1}(X)$,
then $\emptyset\neq(\varphi-\id_{\N_{n-1}(X)})^{-1}(C)\subseteq C$.
\end{itemize}
\end{theorem}

The following lemmas are easy applications but indispensable for us to run equivariant MMP step by step.

\begin{lemma}\label{lem-int-des1} Let $\pi:X\to Y$ be a surjective morphism of projective varieties.
Let $f:X\to X$ and $g:Y\to Y$ be two surjective endomorphisms such that $g\circ\pi=\pi\circ f$.
Suppose $f$ is int-amplified. Then $g$ is int-amplifed.
\end{lemma}
\begin{proof} By Theorem \ref{main-thm-cri}, all the eigenvalues of $f^*|_{\N^1(X)}$ are of modulus greater than $1$ and hence so are all the eigenvalues of $g^*|_{\N^1(Y)}$ since $\pi^*:\N^1(Y)\to \N^1(X)$ is injective.
By Theorem \ref{main-thm-cri} again, $g$ is int-amplified.
\end{proof}

\begin{lemma}\label{lem-int-des2} Let $\pi:X\dasharrow Y$ be a generically finite dominant rational map of projective varieties.
Let $f:X\to X$ and $g:Y\to Y$ be two surjective endomorphisms such that $g\circ\pi=\pi\circ f$.
Then $f$ is int-amplified if and only if so is $g$.
\end{lemma}
\begin{proof}
Let $\Gamma$ be the graph of $\pi$ and denote by $p_X:\Gamma\to X$ and $p_Y:\Gamma\to Y$ be two projections.
Then there exists a surjective endomorphism $h:\Gamma\to \Gamma$ such that $p_X\circ h=f\circ p_X$ and $p_Y\circ h=g\circ p_Y$.
Note that $p_X$ and $p_Y$ are generically finite surjective morphisms.
Therefore, it suffices for us to consider the case when $\pi$ is a well-defined morphism.

One direction follows from Lemma \ref{lem-int-des1}.
Suppose $H:=g^*L-L$ is ample for some ample Cartier divisor $L$ on $Y$.
Then $\pi^*L$ is big and $f^*(\pi^*L)-\pi^*L=\pi^*H$ is big.
By Theorem \ref{main-thm-cri}, $f$ is int-amplified.
\end{proof}

\begin{proof}[Proof of Theorem \ref{main-thm-comp}]
Fix a norm on $\N^1(X)$.
Denote by $\varphi_f:=f^*|_{\N^1(X)}$ and $\varphi_g:=g^*|_{\N^1(X)}$.
Since $f$ is int-amplified, all the eigenvalues of $\varphi_f^{-1}$ are of modulus less than $1$ by Theorem \ref{main-thm-cri}.
Then $\lim\limits_{i\to+\infty}||\varphi_f^{-i}||^{\frac{1}{i}}<1$ and hence there exists some $i_0>0$, such that $||\varphi_f^{-i}||<\frac{1}{||\varphi_g^{-1}||}$ for $i\ge i_0$. 
Denote by $h=f^i\circ g$ with $i\ge i_0$ and $\varphi_h:=h^*|_{\N^1(X)}$.
Let $\frac{1}{r}$ be the spectral radius of $\varphi_h^{-1}$.
By a version of the Perron-Frobenius theorem (cf. \cite{Bi}),
$\varphi_h(v)=rv$ for some nonzero $v\in \N^1(X)$.
Note that $r||v||=||\varphi_h(v)||=||\varphi_g(\varphi_f^i(v))||>||v||$.
So $r>1$ and hence $h$ is int-amplified by Theorem \ref{main-thm-cri} again.
The similar argument works for $g\circ f^i$.
\end{proof}

\begin{proof}[Proof of Theorem \ref{main-thm-k}]
Denote by $\varphi:=f^*|_{\N_{n-1}(X)}$ and $C$ the cone of classes of effective Weil-$\R$ divisors in $\N_{n-1}(X)$.
Then $\varphi(C)=C$.
By the ramification divisor formula, we have the class $[f^*(-K_X)-(-K_X)]=[R_f]\in C$.
So Theorem \ref{thm-cri2} implies that $[-K_X]\in C$.
When $K_X$ is $\Q$-Cartier, the proof is similar.
\end{proof}


Let $f:X\to X$ be a surjective endomorphism of a projective variety $X$ of dimension $n>0$.
Denote by $$\N_i^{\mathbb{C}}(X):=\N_i(X)\otimes_{\R}\mathbb{C}$$ and 
$$\N^k_{\mathbb{C}}(X):=\{\sum a x_1\cdots x_k\,|\, a\in\mathbb{C}, x_1,\cdots, x_k \text{ are Cartier divisors}\}/\equiv_w,$$
where $\sum a x_1\cdots x_k\equiv_w 0$ if $(\sum a x_1\cdots x_k)\cdot x_{k+1}\cdots x_n=0$ for any Cartier divisors $x_{k+1},\cdots, x_n$.
When $k=1$, $\N^1_{\mathbb{C}}(X)=\N^1(X)\otimes_{\mathbb{R}}\mathbb{C}$ by \cite[Lemma 3.2]{Zh-tams}.
Note that $f^*$ naturally induces an invertible linear map on $\N^k_{\mathbb{C}}(X)$.

The following result gives a useful bound on the spectral radius of $f^*|_{\N^k_{\mathbb{C}}(X)}$ for int-amplified $f$ which allows us to discuss the dynamics on the subvarieties later.

\begin{lemma}\label{lem-weak-e} Let $f:X\to X$ be an int-amplified endomorphism of a projective variety $X$ of dimension $n$.
Assume that $0<k<n$.
Then all the eigenvalues of $f^*|_{\N^k_{\mathbb{C}}(X)}$ are of modulus less than $\deg f$.
In particular, $\lim\limits_{i\to +\infty} \frac{(f^i)^*x}{(\deg f)^i}=0$ for any $x\in \N^k_{\mathbb{C}}(X)$.
\end{lemma}
\begin{proof}
We show by induction on $k$ from $n-1$ to $1$.
Suppose $f^*x\equiv_w\mu x$ for some $\mu\ne 0$ and $0\neq x\in \N^k_{\mathbb{C}}(X)$.
Let $V:=\{v\in \N^1_{\mathbb{C}}(X)\,|\, x\cdot v\equiv_w 0\}$ be a subspace of $\N^1_{\mathbb{C}}(X)$.
By the projection formula, $f^*(V)=V$ and $V\subsetneq \N^1_{\mathbb{C}}(X)$.
So there exists some $y\in \N^1_{\mathbb{C}}(X)\backslash V$, such that $f^*y-\lambda y\in V$, where $\lambda$ is an eigenvalue of $f^*|_{\N^1_{\mathbb{C}}(X)}$.
Then $f^*(x\cdot y)\equiv_w \mu\lambda x\cdot y$ and hence $\mu\lambda$ is an eigenvalue of $f^*|_{\N^{k+1}_{\mathbb{C}}(X)}$.
By Theorem \ref{main-thm-cri}, $|\lambda|>1$.
If $k=n-1$, then $\mu\lambda=\deg f$ and hence $|\mu|<\deg f$.
If $k<n-1$, then $|\mu\lambda|<\deg f$ by induction and hence $|\mu|<\deg f$.

The last statement is clear.
\end{proof}

As an easy application, we can show that int-amplified endomorphisms are always non-isomorphic.

\begin{lemma}\label{lem-deg>1} Let $f:X\to X$ be an int-amplified endomorphism of a positive dimensional projective variety $X$.
Then $\deg f>1$.
\end{lemma}
\begin{proof} It is trivial when $\dim(X)=1$.
Assume that $\dim(X)>1$.
By Theorem \ref{main-thm-cri}, all the eigenvalue of $f^*|_{\N^1(X)}$ are of modulus greater than $1$.
Therefore, $\deg f>1$ by applying Lemma \ref{lem-weak-e} for $k=1$.
\end{proof}

Another application is the following lemma concerning the action $f^*|_{\N_k(X)}$ which plays an essential role in our generalization about the singularity and equivariant MMP.
\begin{lemma}\label{lem-lim-int} Let $f:X\to X$ be an int-amplified endomorphism of a projective variety $X$ of dimension $n>0$.
Let $Z$ be a $k$-cycle of $X$ with $k<n$.
Let $H$ be an ample Cartier divisor on $X$.
Then $\lim\limits_{i\to +\infty} Z\cdot\frac{(f^i)^*(H^k)}{(\deg f)^i}=0$.
\end{lemma}
\begin{proof} We may assume $Z$ is effective.
If $k=0$, $\lim\limits_{i\to +\infty} Z\cdot\frac{(f^i)^*(H^k)}{(\deg f)^i}=\lim\limits_{i\to +\infty} \frac{|Z|}{(\deg f)^i}=0$ by Lemma \ref{lem-deg>1}.
Suppose $k>0$.
Let $x_i:=\frac{(f^i)^*(H^k)}{(\deg f)^i}\in \N^k_{\mathbb{C}}(X)$.
By Lemma \ref{lem-weak-e}, $\lim\limits_{i\to +\infty} x_i=0$ in $\N^k_{\mathbb{C}}(X)$.
Since $H$ is ample, $x_i\cdot w_e\ge 0$ for any effective $k$-cycle $w_e$. 
So the lemma follows from Lemma \ref{lem-wts}.
\end{proof}

\begin{lemma}\label{lem-wts} Let $X$ be a projective variety of dimension $n$. Suppose $x_i\in \N^k_{\mathbb{C}}(X)$  with $0<k<\dim(X)$ such that $x\cdot w_e\ge 0$ for any non-zero effective $k$-cycle $w_e$.
Suppose further $\lim\limits_{i\to +\infty} x_i=0$ in $\N^k_{\mathbb{C}}(X)$.
Then for any $w\in \N_i^{\mathbb{C}}(X)$, $\lim\limits_{i\to +\infty} x_i\cdot w=0$.
\end{lemma}
\begin{proof} We may assume that $w$ represents the class of some irreducible closed subvariety. By Lemma \ref{lem-int-hyp}, $w+w'=y_1\cdots y_{n-k}$ for some effective $k$-cycle $w'$ and hypersurfaces $y_1,\cdots, y_{n-k}$.
So $0\le \lim\limits_{i\to +\infty} x_i\cdot w\le \lim\limits_{i\to +\infty} x_i\cdot(w+w')=\lim\limits_{i\to +\infty} x_i\cdot y_1\cdots y_{n-k}=0$.
\end{proof}

\begin{lemma}\label{lem-int-hyp} Let $X$ be a projective variety of dimension $n$.
Let $W$ be an $m$-dimensional closed subvariety of $X$ with $m<n$.
Then there exist hypersurfaces $H_1,\cdots, H_{n-m}$ such that $\bigcap_{i=1}^{n-m}H_i$ is of pure dimension $m$ and $W$ is an irreducible component of $\bigcap_{i=1}^{n-m}H_i$.
In particular, the intersection $H_1\cdots H_{n-m}=W+W'$ for some effective $m$-cycle $W'$.
\end{lemma}
\begin{proof} Let $X$ be a closed subvariety of $\mathbb{P}^N$ for some $N>0$.
Let $I$ be the homogeneous ideal of $W$ in $\mathbb{P}^N$.

Let $Y_1,\cdots, Y_s$ be irreducible closed subvarieties of $\mathbb{P}^N$ such that $W$ does not contain $Y_i$ for each $i$.
We first claim that there exists a homogenous polynomial $f\in I$ such that $Z(f)$ (zeros of $f$, not necessarily irreducible or reduced) does not contain $Y_i$ for each $i$.
Since $W$ does not contain $Y_1$, there exists some homogenous polynomial $f_1\in I$ such that $Z(f_1)$ does not contain $Y_1$.
Suppose we have found some homogenous polynomial $f_t\in I$ such that $Z(f_t)$ does not contain $Y_i$ for $i\le t$.
Since $W$ does not contain $Y_{t+1}$, there exists some homogenous polynomial $g_{t+1}\in I$ such that $Z(g_{t+1})$ does not contain $Y_{t+1}$.
We may assume $f_t$ and $g_{t+1}$ have the same degree by taking suitable powers.
If $Z(f_{t})$ does not contain $Y_{t+1}$, we set $f_{t+1}=f_t$.
Suppose $Y_{t+1}\subseteq Z(f_{t+1})$.
Let $k$ be the base field.
Consider $S_i:=\{a\in k\,|\, Y_{i}\subseteq Z(f_t+a g_{t+1}) \}$ for $i\le t+1$.
Note that $S_i$ has at most one element for each $i\le t+1$.
Since $k$ is infinite, there exists some $a\not\in \bigcup_{i=1}^{t+1}S_i$ and we set $f_{t+1}=f_t+a g_{t+1}$.
So the claim is proved.

By the above claim, we may first find some homogenous polynomial $h_1\in I$ such that $Z(h_1)$ does not contain $X$. Set $H_1:=Z(h_1)|_X$ the pullback of $Z(h_1)$ via the inclusion map $X\to \mathbb{P}^N$.
Then $H_1$ is a hypersurface of $X$.
Suppose that we have found hypersurfaces $H_1:=Z(h_1)|_X,\cdots, H_t:=Z(h_t)|_X$ such that $\bigcap_{i\le t}H_t$ is of pure dimension $n-t$ and contains $W$.
If $m=n-t$, then $W$ is an irreducible component of $\bigcap_{i\le t}H_t$ and hence we are done.
If $m<n-t$, then we may continue applying the above claim to all the irreducible components of $\bigcap_{i\le t}H_t$.
\end{proof}

Applying Lemma \ref{lem-lim-int} to the $f^{-1}$-invariant closed subvariety, we have
\begin{lemma}\label{lem-inv-deg} Let $f:X\to X$ be an int-amplified endomorphism of a projective variety $X$.
Let $Z$ be an $f^{-1}$-invariant closed subvariety of $X$ such that $\deg f|_Z=\deg f$.
Then $Z=X$.
\end{lemma}
\begin{proof}
Let $m:=\dim(Z)$ and $d:=\deg f$.
Suppose $m<\dim(X)$.
Let $A$ be an ample Cartier divisor on $X$.
Then $Z\cdot f^*(A)^m=f_*Z\cdot A^m=dZ\cdot A^m$ by the projection formula.
By Lemma \ref{lem-lim-int}, we have $1\le Z\cdot A^m=\lim\limits_{i\to+\infty} Z\cdot\frac{(f^i)^*(A^m)}{d^i}=0$, a contradiction.
\end{proof}

Now we are able to apply \cite[Theorem 1.2]{BH} and show the singularity.
\begin{proof}[Proof of Theorem \ref{main-thm-lc}]
Suppose the contrary that $X$ is not lc.
Let $Z$ be an irreducible component the non-lc locus of $X$.
Since $f$ is int-amplified, $\deg f>1$ by Lemma \ref{lem-deg>1}.
Then $f^{-1}(Z)=Z$ and $\deg f|_Z=\deg f>1$ by \cite[Theorem 1.2]{BH}.
By Lemma \ref{lem-inv-deg}, $Z=X$, absurd.
\end{proof}

\section{$Q$-abelian case}
In this section, we will deal with the case of $Q$-abelian varieties admitting int-amplified endomorphisms.
Recall that a normal projective variety $X$ is $Q$-abelian if there exist an abelian variety $A$ and a finite surjective morphism $\pi:A\to X$ which is \'etale in codimension $1$.
As stated in Theorem \ref{main-thm-mmp}, $Q$-abelian varieties are the end products of the equivariant MMP and this will be proved in the next section.
The results discussed in this section will be used to show many rigidities, eg., Theorem \ref{main-thm-mrc}, Theorem \ref{main-thm-alb} and Theorem \ref{main-thm-mmp}(2).

First, we observe in the following two lemmas that there is no $f^{-1}$-periodic subvarieity except itself.
\begin{lemma}\label{lem-inv-abe}
Let $f:A\to A$ be an int-amplified endomorphism of an abelian variety $A$.
Let $Z$ be a (non-empty) $f^{-1}$-periodic closed subvariety of $A$.
Then $Z=A$.
\end{lemma}
\begin{proof}
We may assume $Z$ is irreducible and $f^{-1}(Z)=Z$.
Note that $f$ is \'etale by the ramification divisor formula and the purity of branch loci.
Then $\deg f|_Z=\deg f$ and hence $Z=A$ by Lemma \ref{lem-inv-deg}.
\end{proof}

\begin{lemma}\label{lem-inv-qabe}
Let $f:X\to X$ be an int-amplified endomorphism of a $Q$-abelian variety.
Let $Z$ be a (non-empty) $f^{-1}$-periodic closed subset of $X$.
Then $Z=X$.
\end{lemma}
\begin{proof}
Note that $X$ has only quotient singularities and $K_X\sim_{\Q}0$.
Then $f$ is quasi-\'etale by the ramification divisor formula.
By \cite[Corollary 8.2]{CMZ}, there exist a quasi-\'etale cover $\pi:A\to X$ and a surjective endomorphism $f_A:A\to A$ such that $f\circ\pi=\pi\circ f_A$.
By Lemma \ref{lem-int-des2}, $f_A$ is int-amplified.
Note that $f_A^{-s}(\pi^{-1}(Z))=\pi^{-1}(Z)$ for some $s>0$.
By Lemma \ref{lem-inv-abe}, $\pi^{-1}(Z)=A$ and hence $Z=X$.
\end{proof}

Now we state several rigidities.
The proof of the following lemma is the same as the proof of \cite[Lemma 5.2]{MZ} except that we apply Lemma \ref{lem-inv-qabe} instead of \cite[Lemma 4.7]{MZ}.
We rewrite it here for the reader's convenience.
\begin{lemma}\label{fibres-rc+irr} Let $\pi:X\to Y$ be a surjective morphism between normal projective varieties with connected fibres.
Let $f:X\to X$ and $g:Y\to Y$ be two int-amplified endomorphisms such that $g\circ\pi=\pi\circ f$.
Suppose that $Y$ is $Q$-abelian.
Then the following are true.
\begin{itemize}
\item[(1)]
All the fibres of $\pi$ are irreducible.
\item[(2)]
$\pi$ is equi-dimensional.
\item[(3)]
If the general fibre of $\pi$ is rationally connected, then all the fibres of $\pi$ are rationally connected.
\end{itemize}
\end{lemma}
\begin{proof} First we claim that $f(\pi^{-1}(y))=\pi^{-1}(g(y))$ for any $y\in Y$.
Suppose there is a closed point $y$ of $Y$ such that $f|_{\pi^{-1}(y)}:\pi^{-1}(y)\to \pi^{-1}(g(y))$ is not surjective.
Let $S=g^{-1}(g(y))-\{y\}$.
Then $S\neq \emptyset$ and $U:=X-\pi^{-1}(S)$ is an open dense subset of $X$.
Since $f$ is an open map, $f(U)$ is an open dense subset of $X$. In particular, $f(U)\cap \pi^{-1}(g(y))$ is open in $\pi^{-1}(g(y))$. Note that $f(U)=(X-\pi^{-1}(g(y)))\cup f(\pi^{-1}(y))$. So $f(U)\cap \pi^{-1}(g(y))=f(\pi^{-1}(y))$ is open in $\pi^{-1}(g(y))$. Since $f$ is also a closed map, the set $f(\pi^{-1}(y))$ is both open and closed in the connected fibre $\pi^{-1}(g(y))$ and hence $f(\pi^{-1}(y))=\pi^{-1}(g(y))$.
So the claim is proved.

Let  $$\Sigma_1:=\{y\in Y\,|\, \pi^{-1}(y) \text{ is not irreducible}\}.$$
Note that $f(\pi^{-1}(y))=\pi^{-1}(g(y))$.
Then $g^{-1}(\Sigma_1)\subseteq \Sigma_1$ and hence  $g^{-1}(\overline{\Sigma_1})\subseteq \overline{\Sigma_1}$.
Since $\overline{\Sigma_1}$ is closed and has finitely many irreducible components, $g^{-1}(\overline{\Sigma_1})=\overline{\Sigma_1}$.
\textbf{By Lemma \ref{lem-inv-qabe}}, $\Sigma_1=\emptyset$.
So (1) is proved.

Let $$\Sigma_2:=\{y\in Y\,|\,\dim(\pi^{-1}(y))>\dim(X)-\dim(Y)\},$$ and
$$\Sigma_3:=\{y\in Y\,|\, \pi^{-1}(y) \text{ is not rationally connected}\}.$$

By (1), $\pi$ is equi-dimensional outside $\Sigma_2$.
Since $f$ is finite surjective, $g^{-1}(\Sigma_2)\subseteq \Sigma_2$.
By (1), all the fibres of $\pi$ outside $\Sigma_3$ are rationally connected.
Note that the image of a rationally connected variety is rationally connected.
So $g^{-1}(\Sigma_3)\subseteq \Sigma_3$.
Now the same reason above implies that $\Sigma_2=\emptyset$.
Similarly, $\Sigma_3=\emptyset$ if the general fibre of $\pi$ is rationally connected.
\end{proof}

We recall the following rigidity without dynamics.
\begin{lemma}\label{mor-q-abelian} (cf.~\cite[Lemma 5.3]{MZ}) Let $\pi:X\dashrightarrow Y$ be a dominant rational map between normal projective varieties.
Suppose that $(X,\Delta)$ is a klt pair for some effective $\mathbb{Q}$-divisor $\Delta$ and $Y$ is $Q$-abelian.
Suppose further that the normalization of the graph $\Gamma_{X/Y}$ is equi-dimensional over $Y$ (this holds when $k(Y)$ is algebraically closed in $k(X)$, $f: X \to X$ is int-amplified and
$f$ descends to some int-amplified $f_Y : Y \to Y$).
Then $\pi$ is a morphism.
\end{lemma}
\begin{proof}
Note that the lemma is the same with \cite[Lemma 5.3]{MZ} except the argument in brackets.
Let $W$ be the normalization of $\Gamma_{X/Y}$ and denote by $p_X:W\to X$ and $p_Y:W\to Y$ be the two projections.
So we are left to prove that the argument in brackets implies that $p_Y$ is equi-dimensional.
In this situation, there exists a surjective endomorphism $h:W\to W$ such that $p_X\circ h=f\circ p_X$ and $p_Y\circ h=f_Y\circ p_Y$.
By Lemma \ref{lem-int-des2}, $h$ is int-amplified.
So $p_Y$ is equi-dimensional by Lemma \ref{fibres-rc+irr}.
\end{proof}

\section{$K_X$ pseudo-effective case}
In this section, we reduce $K_X$ pseudo-effective case to the $Q$-abelian case.
In this way, we are only left to deal with the case when $K_X$ is not pseudo-effective and the equivariant MMP; see Section \ref{sec-mmp}. 

We first recall the result below.
\begin{lemma}\label{lem-nak-gkp}(cf.~\cite[Lemma 3.3.1]{ENS}, \cite[Lemma 2.5]{Na-Zh} and \cite[Theorem 1.1]{GKP})
Let $f:X\to X$ be a non-isomorphic surjective endomorphism of a normal projective variety $X$.
Let $\theta_k:V_k\to X$ be the Galois closure of $f^k:X\to X$ for $k\ge 1$ and let $\tau_k:V_k\to X$ be the induced finite Galois covering such that $\theta_k=f^k\circ \tau_k$.
Then there exist finite Galois morphims $g_k, h_k: V_{k+1}\to V_k$ such that $\tau_k\circ g_k=\tau_{k+1}$ and $\tau_k\circ h_k=f\circ \tau_{k+1}$.
Suppose further that $X$ is klt and $f$ is quasi-\'etale.
Then $g_k$ and $h_k$ are \'etale when $k\gg 1$.
\end{lemma}

For the result below, we follow the idea of \cite[Theorem 3.3]{Na-Zh} and rewrite the proof here.
\begin{theorem}\label{thm-kx-pe-qa} Let $f:X\to X$ be an int-amplified endomorphism of a klt normal projective variety $X$ with $K_X$ being pseudo-effective.
Then $X$ is $Q$-abelian.
\end{theorem}
\begin{proof} 
By Theorem \ref{main-thm-k}, $-K_X$ is pseudo-effective and hence $K_X\equiv 0$.
Therefore, $f$ is quasi-\'etale by the ramification divisor formula.
We then apply Lemma \ref{lem-nak-gkp} and use the notation there.
Note that $\deg h_k=d(\deg g_k)$ where $d:=\deg f>1$ by Lemma \ref{lem-deg>1}.
Let $A$ be an ample Cartier divisor on $X$.
Denote by $A_{k}:=\tau_{k}^*A$ and $(f^*A)_{k}:=\tau_{k}^*(f^*A)$.
In the rest of the proof, we always assume $k\gg 1$.

We first claim that $V_k$ is smooth.
Denote by $\Sing(V_k)$ the singular locus of $V_k$.
Note that $g_k$ and $h_k$ are \'etale and Galois.
So we may assume $\Sing(V_{k+1})= g_k^{-1}(\Sing(V_k))= h_k^{-1}(\Sing(V_k))$.
Suppose the contrary $\Sing(V_k)\neq \emptyset$.
Let $m:=\dim (\Sing(V_k))<\dim(X)$.
Let $S_k$ be the union of the $m$-dimensional irreducible components of $\Sing(V_k)$.
We may assume $S_{k+1}=g_k^{-1}(S_k)=h_k^{-1}(S_k)$ and $S_{k+1}=g_k^*S_k=h_k^*S_k$ as cycles.
By the projection formula, we have
$$S_{k+1}\cdot(f^*A)_{k+1}^m=S_{k+1}\cdot g_k^*((f^*A)_k)^m=(\deg g_k)  S_k\cdot (f^*A)_k^m$$
and
$$S_{k+1}\cdot(f^*A)_{k+1}^m=S_{k+1}\cdot h_k^*(A_k)^m=(\deg h_k) S_k\cdot A_k^m.$$
Then $S_k\cdot (f^*A)_k^m=d S_k\cdot A_k^m$.
Let $Z_k:=(\tau_k)_{\ast}S_k$.
By the projection formula, we have $Z_k\cdot (f^*A^m)=d Z_k\cdot A^m$.
Therefore, $1\le Z_k\cdot A^m=\lim\limits_{i\to+\infty}Z_k\cdot \frac{(f^i)^*A^m}{d^i}=0$ by Lemma \ref{lem-lim-int}, a contradiction.
So the claim is proved.

Let $n:=\dim(X)$.
Next, we claim that $c_2(V_k)\cdot A_k^{n-2}=0$.
Note that $c_2(V_{k+1})=g_k^*(c_2(V_k))=h_k^*(c_2(V_k))$.
By a similar argument, 
we have $$c_2(V_{k+1})\cdot (f^*A)_{k+1}^{n-2}=(\deg g_k)c_2(V_{k})\cdot (f^*A)_k^{n-2}=(\deg h_k)c_2(V_{k})\cdot A_k^{n-2}.$$
Let $W_k:=(\tau_k)_{\ast}c_2(V_k)$.
By the projection formula, $c_2(V_k)\cdot (f^*A)_k^{n-2}=W_k\cdot (f^*A^{n-2})=
dc_2(V_k)\cdot A_k^{n-2}=dW_k\cdot A^{n-2}$.
Therefore $c_2(V_k)\cdot A_k^{n-2}=W_k\cdot A^{n-2}=\lim\limits_{i\to+\infty} W_k\cdot\frac{(f^i)^*A^{n-2}}{d^i}=0$ by Lemma \ref{lem-lim-int}.

Since $f^k$ is quasi-\'etale, its Galois closure $\theta_k$ is quasi-\'etale and hence so is $\tau_k$.
In particular, $c_1(V_k)$ is numerically trivial.
Therefore, $V_k$ is $Q$-abelian by \cite{Ya} (cf.~\cite{Be}) and hence
so is $X$.
\end{proof}

\section{Special MRC fibration and the non-uniruled case}\label{sec-mrc}
In this section, we apply Theorem \ref{thm-kx-pe-qa} to reduce the non-uniruled case to the $Q$-abelian case and prove Theorem \ref{main-thm-mrc}.

We slightly generalize \cite[Lemma 2.4]{HMZ} to the following.
\begin{lemma}\label{lem-nu-kx0}
Let $X$ be a non-uniruled normal projective variety such that $-K_X$ is pseudo-effective.
Then $K_X\sim_{\Q}0$ ($\Q$-linear equivalence) and $X$ has only canonical singularities.
\end{lemma}
\begin{proof}
Let $\pi:Y\to X$ be a resolution of $X$.
Since $Y$ is non-uniruled, $K_Y$ is pseudo-effective by \cite[Theorem 2.6]{BDPP}.
Thus, we have the $\sigma$-decomposition $K_Y=P_\sigma(K_Y)+N_\sigma(K_Y)$ in the sense of \cite{ZDA}: $N_\sigma(K_Y)$ is an effective $\R$-Cartier divisor determined by the following property:
$P_\sigma(K_Y) = K_Y - N_\sigma(K_Y)$ is movable,
and if $B$ is an effective $\R$-Cartier divisor such that $K_Y - B$ is movable, then $N_\sigma(K_Y)\le B$.
Here, an $\R$-Cartier divisor $D$ is called movable if:
for any ample $\R$-Cartier divisor $H'$ and any prime divisor $\Gamma$, there is an effective $\R$-Cartier divisor $\Delta$ such that
$\Delta\sim D+ H'$ and $\Gamma\not\subset \Supp \Delta$ (cf.~\cite[Chapter III, \S 1.b]{ZDA}).

Note that $K_X=\pi_*K_Y\sim \pi_*P_\sigma(K_Y)+\pi_*N_\sigma(K_Y)$ and $-K_X$ is pseudo-effective.
We have $K_X\equiv_w 0$ (weak numerical equivalence, cf.~\cite[\S 2]{MZ}).
Then $\pi_*P_\sigma(K_Y)\equiv_w 0$.
Let $H$ be an ample Cartier divisor on $X$ and $n:=\dim(X)$.
Then $P_\sigma(K_Y)\cdot (\pi^*H)^{n-1}=0$ by the projection formula.
Since $P_\sigma(K_Y)$ is movable and $\pi^*H$ is nef and big, $P_\sigma(K_Y)\equiv 0$ by Lemma \ref{lem-mov-0}.
In particular, the numerical Kodaira dimension $\kappa_\sigma(Y)$ of $Y$, in the sense of \cite[Chapter V]{ZDA}, is zero.
By \cite[Corollary 4.9]{ZDA}, the Kodaira dimension $\kappa(Y)=0$.
Therefore, $K_Y\sim_{\Q} E$ for some effective $\Q$-Cartier divisor $E$.
Note that $\pi_*E\sim_{\Q}\pi_*K_Y\equiv_w 0$.
Then $E$ is $\pi$-exceptional and hence $K_X\sim_{\Q} 0$.

Note that $K_Y\sim_{\Q}\pi^*K_X+E$.
So $X$ has canonical singularities.
\end{proof}

\begin{lemma}\label{lem-mov-0}
Let $X$ be a smooth projective variety of dimension $n$.
Let $D$ be a movable $\R$-Cartier divisor such that $D\cdot H^{n-1}=0$ for some nef and big Cartier divisor $H$.
Then $D\equiv 0$.
\end{lemma}
\begin{proof}
Since $D$ is movable, we may write $D\equiv \lim\limits_{i\to +\infty} D_i$ where $m_iD_i$ is Cartier and effective for some $m_i$ and the base locus of $|m_iD_i|$ is of codimension at least 2. 
Then for any prime divisor $G$, $D_i|_G$ is effective and hence $D|_G\equiv \lim\limits_{i\to +\infty} D_i|_G$ is pseudo-effective.
So $D\cdot G\cdot H^{n-2}=D|_G\cdot (H|_G)^{n-2}\ge 0$.
By \cite[Lemma 2.2]{Na-Zh}, $D\equiv 0$ .
\end{proof}

\begin{theorem}\label{thm-nonuniruled} Let $f:X\to X$ be an int-amplified endomorphism of a non-uniruled normal projective variety $X$.
Then $X$ is $Q$-abelian with only canonical singularities.
\end{theorem}
\begin{proof} By Theorem \ref{main-thm-k}, $-K_X$ is pseudo-effective.
So $K_X\sim_{\Q}0$ and $X$ has only canonical singularities by Lemma \ref{lem-nu-kx0}.
In particular, $K_X$ is pseudo-effective and hence $X$ is $Q$-abelian by Theorem \ref{thm-kx-pe-qa}.
\end{proof}

We recall the following definition for the reader's convenience.
\begin{definition}[Special MRC fibration]\label{def-smrc}
Let $X$ be a normal projective variety. 
A {\it special MRC fibration} for $X$ is a dominant rational map $\pi:X\dasharrow Y$ into a normal projective variety $Y$ such that
\begin{itemize}
\item[(1)] $Y$ is non-uniruled,
\item[(2)] the second projection $p_2: \Gamma \to Y$ for the graph $\Gamma\subseteq X \times Y$ of $\pi$ is equi-dimensional,
\item[(3)] a general fiber of $p_2$ is rationally connected,
\item[(4)] $\pi$ is a Chow reduction.
\end{itemize}
The special MRC fibration always exists and $\pi$ is uniquely determined up to isomorphism (cf.~\cite[Theorem 4.18]{Na10}).
Moreover, any surjective endomorphism $f:X\to X$ descends to some surjective endomorphism $g:Y\to Y$ equivariantly via $\pi$ (cf.~\cite[Theorem 4.19]{Na10}).
\end{definition}

\begin{proof}[Proof of Theorem \ref{main-thm-mrc}]
(1) follows from \cite[Theorem 4.19]{Na10} (cf.~\cite[Lemma 4.1]{MZ}) and Lemma \ref{lem-int-des1}.
(2) follows from Theorem \ref{thm-nonuniruled}.
(3) follows from Lemma \ref{fibres-rc+irr}.
(4) follows from Lemma \ref{mor-q-abelian}.
\end{proof}

\section{Albanese morphism and Albanese map}\label{sec-alb}
In this section, we prove Theorem \ref{main-thm-alb}.

We recall the notion of Albanese morphism and Albanese map of a normal projective variety (cf.~\cite[\S5]{CMZ}).
\begin{definition}\label{def-alb} Let $X$ be a normal projective variety.

There is an {\it Albanese morphism} $\alb_X:X\to \Alb(X)$ such that:
$\Alb(X)$ is an abelian variety,
$\alb_X(X)$ generates $\Alb(X)$, and for every morphism $\varphi:X\to A$ from $X$ to an abelian variety $A$,
there exists a unique morphism $\psi:\Alb(X)\to A$ such that $\varphi = \psi\circ \alb_X$
(cf. \cite[Remark 9.5.25]{FGI}).

In the birational category, there is an {\it Albanese map} $\mathfrak{alb}_X:X\dashrightarrow \mathfrak{Alb}(X)$ such that:
$\mathfrak{Alb}(X)$ is an abelian variety, $\mathfrak{alb}_X(X)$ generates $\mathfrak{Alb}(X)$, and for every rational map $\varphi:X\dashrightarrow A$ from $X$ to an abelian variety $A$,
there exists a unique morphism $\psi:\mathfrak{Alb}(X)\to A$ such that $\varphi = \psi\circ \mathfrak{alb}_X$ (cf.~\cite[Chapter II.3]{Lang}).
If $\mathfrak{alb}_X$ is a morphism, then $\mathfrak{alb}_X$ and $\alb_X$ are the same.

Let $f:X\to X$ be a surjective endomorphism of a normal projective variety $X$ over $k$.
By the above two universal properties, $f$ descends to surjective endomorphisms on $\Alb(X)$ and $\mathfrak{Alb}(X)$.
\end{definition}

\begin{lemma}\label{lem-aa}
Let $f:A\to A$ be a surjective endomorphism of an abelian variety $A$.
Let $Z$ be an $f$-invariant subvariety of $A$ such that $f|_Z$ is amplified.
Then $Z$ is an abelian variety.
\end{lemma}
\begin{proof}
By Theorem \ref{thm-amp-kod}, $\kappa(Z)\le 0$.
Therefore, $Z$ is an abelian variety by \cite[Theorem 3.10]{Ue}.
\end{proof}

\begin{lemma}\label{lem-alb-surj}
Let $f:X\to X$ be an int-amplified endomorphisms of a normal projective variety $X$.
Then the Albanese morphism $\alb_X$ is surjective.
\end{lemma}
\begin{proof}
Let $Z:=\alb(X)$ and $g:=f|_{\Alb(X)}$.
Then $g(Z)=Z$.
By Lemma \ref{lem-int-des1}, $g|_Z$ is int-amplified and hence amplified.
By Lemma \ref{lem-aa}, $Z$ is an abelian variety.
By the universal property of $\alb_X$, we have $Z=\Alb(X)$.
\end{proof}

\begin{lemma}\label{lem-alb-fin-iso} Let $f:X\to X$ be an int-amplified endomorphism of a normal projective variety $X$.
Suppose $\alb_X$ is finite.
Then $\alb_X$ is an isomorphism and $X$ is an abelian variety.
\end{lemma}
\begin{proof}
By Lemma \ref{lem-alb-surj}, $\alb_X$ is surjective.
By the ramification divisor formula, $K_X=(\alb_X)^*K_{\Alb(X)}+R_{\alb_X}=R_{\alb_X}$, where $R_{\alb_X}$ is the effective ramification divisor of $\alb_X$.
By Theorem \ref{main-thm-k}, $-K_X$ is pseudo-effective.
So $R_{\alb_X}=0$ and hence $\alb_X$ is \'etale by the purity of branch loci.
Therefore, $X$ is an abelian variety (cf.~\cite[Chapter IV, 18]{Mu}).
By the universal property, $\alb_X$ is an isomorphism.
\end{proof}

\begin{proof}[Proof of Theorem \ref{main-thm-alb}]
Let $g:=f|_{\Alb(X)}$.
By Lemmas \ref{lem-alb-surj} and \ref{lem-int-des1}, $\alb_X$ is surjective and $g$ is int-amplified.
Taking the Stein factorization of $\alb_X$, we have $\varphi:X\to Y$ and $\psi:Y\to \Alb(X)$ such that $\varphi_\ast\mathcal{O}_X=\mathcal{O}_Y$ and $\psi$ is a finite morphism.
Then $f$ descends to an int-amplified endomorphism $f_Y:Y\to Y$ by \cite[Lemma 5.2]{CMZ} and Lemma \ref{lem-int-des1}.
By the universal property, $\psi=\alb_Y$.
So $\psi$ is an isomorphism by Lemma \ref{lem-alb-fin-iso}, and we can identify
$\alb_X : X \to \Alb(X)$ with $\varphi: X \to Y$.
By Lemma \ref{fibres-rc+irr}, all the fibres of $\alb_X$ are irreducible and equi-dimensional.
So (1) is proved.

For (2), let $W$ be the normalization of the graph of $\mathfrak{alb}_X$.
Then $\Alb(W)=\mathfrak{Alb}(W)=\mathfrak{Alb}(X)$.
Note that $f$ lifts to an int-amplified endomorphism $f_W:W\to W$ by Lemma \ref{lem-int-des2}.
Therefore by (1), the induced endomorphism $h:\mathfrak{Alb}(X)\to \mathfrak{Alb}(X)$ is int-amplified.
Since $\alb_W$ is surjective by (1), $\mathfrak{alb}_X$ is dominant.
\end{proof}

\section{Minimal model program for int-amplified endomorphisms}\label{sec-mmp}
In this section, we apply Lemma \ref{lem-lim-int} and generalize the theory of equivariant MMP to the int-amplified case.
We refer to \cite[Section 6]{MZ} for all technical details involved.

We rewrite the proof of \cite[Lemma 6.1]{MZ} by highlighting the differences for the reader's convenience.
\begin{lemma}\label{lem-MA} Let $f:X\to X$ be an int-amplified endomorphism of a projective variety.
Suppose $A\subseteq X$ is a closed subvariety with $f^{-i}f^i(A) = A$ for
all $i\ge 0$. Then $M(A) := \{f^i(A)\,|\, i \in \Z\}$ is a finite set.
\end{lemma}
\begin{proof}
We may assume $n := \dim(X) \ge 1$.
Set $M_{\ge0}(A):=\{f^i(A)\,|\, i \ge 0\}$.

We first assert that if $M_{\ge 0}(A)$ is a finite set, then so is $M(A)$.
Indeed, suppose $f^{r_1}(A)=f^{r_2}(A)$ for some $0<r_1<r_2$.
Then for any $i>0$, $f^{-i}(A)=f^{-i}f^{-sr_1}f^{sr_1}(A)=f^{-i}f^{-sr_1}f^{sr_2}(A)=f^{sr_2-sr_1-i}(A)\in M_{\ge 0}(A)$ if $s\gg 1$.
So the assertion is proved.

Next, we show that $M_{\ge0}(A)$ is a finite set by induction on the codimension of $A$ in $X$.
We may assume $k := \dim(A)<\dim(X)$.
Let $\Sigma$ be the union of $\Sing (X)$, $f^{-1}(\Sing (X))$ and the irreducible components in the ramification divisor $R_f$ of $f$.
Set $A_i := f^i(A) \, (i\ge 0)$.

\textbf{We claim that $A_i$ is contained in $\Sigma$ for infinitely many $i$.}
Otherwise, replacing $A$ by some $A_{i_0}$, we may assume that $A_i$ is not contained in $\Sigma$ for all $i\ge 0$.
So we have $f^{\ast}A_{i+1}= A_i$.
Let $H$ be an ample Cartier divisor.
By the projection formula, $A_{i+1}\cdot H^k=A_i\cdot(\frac{1}{d}f^*(H^k))\ge 1$.
By Lemma \ref{lem-lim-int}, $1\le \lim\limits_{i\to +\infty}A_{i+1}\cdot H^k=\lim\limits_{i\to +\infty}A_1\cdot (\frac{1}{d^i}(f^i)^*(H^k))=0$, a contradiction.
So the claim is proved.

If $k=n-1$, by the claim, $f^{r_1}(A)=f^{r_2}(A)$ for some $0<r_1<r_2$.
Then $|M_{\ge 0}(A)|<r_2$.

If $k\leq n-2$, assume that $|M_{\ge 0}(A)|=\infty$.
Let $B$ be the Zariski-closure of the union of those $A_{i_1}$ contained in $\Sigma$.
Then $k+1\leq \dim(B)\le n-1$, and $f^{-i}f^i(B) = B$ for all $i \ge 0$. Choose $r \ge 1$ such that $B' := f^r(B), f(B'), f^2(B'), \cdots$ all have the same number of irreducible components.
Let $X_1$ be an irreducible component of $B'$ of maximal dimension.
Then $k+1\leq \dim(X_1)\le n-1$ and $f^{-i}f^i(X_1) = X_1$ for all $i \ge 0$. By induction, $M_{\ge 0}(X_1)$ is a finite set.
So we may assume that $f^{-1}(X_1)=X_1$, after replacing $f$ by a positive power and $X_1$ by its image.
\textbf{Note that $f|_{X_1}$ is still int-amplified by Lemma \ref{lem-amp-res}.}
Now the codimension of $A_{i_1}$ in $X_1$ is smaller than that of $A$ in $X$.
By induction, $M_{\ge 0}(A_{i_1})$ and hence $M_{\ge 0}(A)$ are finite.
\end{proof}

\begin{theorem}\label{thm-equi-mmp} Let $f:X\to X$ be an int-amplified endomorphism of a $\Q$-factorial lc projective variety $X$.
Let $\pi:X\dasharrow Y$ be a dominant rational map which is either a divisorial contraction or a Fano contraction or a flipping contraction or a flip induced by a $K_X$-negative extremal ray.
Then there exists an int-amplified endomorphism $g:Y\to Y$ such that $g\circ\pi=\pi\circ f$ after replacing $f$ by a positive power.
\end{theorem}
\begin{proof} Replacing \cite[Lemma 6.1]{MZ} by our new Lemma \ref{lem-MA}, then the theorem follows by the same argument and proofs of \cite[Lemma 6.2 to Lemma 6.6]{MZ}.
\end{proof}

\section{Proof of Theorems \ref{main-thm-mmp} and \ref{main-thm-rc-diag}}\label{sec-proof}

%

Let $X$ be a $\Q$-factorial lc projective variety.
Let $\pi:X\to Y$ be a contraction of a $K_X$-negative extremal ray $R_C:=\mathbb{R}_{\ge0}C$ generated by some curve $C$.
Then $\NS(X)/\pi^*\NS(Y)$ is a $\mathbb{Z}$-module of rank $1$ by
the exact sequence (cf.~\cite[Theorem 1.1(4)iii]{Fu11}, or \cite[Corollary 3.17]{KM}) below
$$0\to \NS(Y)\xrightarrow{\pi^{\ast}}\NS(X)\xrightarrow{\cdot C} \mathbb{Z}\to 0.$$

Tensoring with $\R$, $\N^1(X)/\pi^*\N^1(Y)$ is a $1$-dimensional real vector space.
Let $D\in \N^1(X)$.
Then $D\cdot C=0$ implies $D\in \pi^*\N^1(Y)$;
$D\cdot C>0$ implies $D$ is $\pi$-ample;
and $D\cdot C<0$ implies $-D$ is $\pi$-ample.

Let $f:X\to X$ be an int-amplified endomorphism.
By Theorem \ref{thm-equi-mmp}, there exists some int-amplified endomorphism $g:Y\to Y$ such that $g\circ\pi=\pi\circ f$.
In particular, we have an induced map $f^*:\NS(X)/\pi^*\NS(Y)\to \NS(X)/\pi^*\NS(Y)$.
Tensoring with $\R$, we have an induced invertible linear map $f^*:\N^1(X)/\pi^*{\N^1(Y)}\to \N^1(X)/\pi^*{\N^1(Y)}$. Note that all the eigenvalues of $f^*|_{\N^1(X)}$ are of modulus greater than $1$ by Theorem \ref{main-thm-cri}. So we have the following.

\begin{lemma}\label{lem-eig-int} Let $X$ be a $\Q$-factorial lc projective variety.
Let $\pi:X\to Y$ be a contraction of a $K_X$-negative extremal ray. Let $f:X\to X$ and $g:Y\to Y$ be int-amplified endomorphisms such that $g\circ \pi=\pi\circ f$.
Then $f^*|_{\N^1(X)/\pi^*\N^1(Y)}=q \id$ for some positive integer $q>1$.
\end{lemma}

\begin{lemma}\label{lem-diag-fano} Let $X$ be a $\Q$-factorial lc projective variety.
Let $\pi:X\to Y$ be a Fano contraction of a $K_X$-negative extremal ray.
Let $f:X\to X$ and $g:Y\to Y$ be surjective endomorphisms such that $g\circ \pi=\pi\circ f$.
Suppose $g^*|_{\N^1_{\mathbb{C}}(Y)}$ is diagonalizable. Then so is $f^*|_{\N^1_{\mathbb{C}}(X)}$.
\end{lemma}

\begin{proof} Let $m:=\dim(X)$, $n:=\dim(Y)$ and $a:=m-n$.
Consider $W:=\N^1_{\mathbb{C}}(Y)$ as a subspace of $V:=\N^1_{\mathbb{C}}(X)$ via the pullback $\pi^*$.
Denote by $V_{\R}:=\N^1(X)$ and $W_{\R}:=\N^1(Y)$.
Let $\varphi:=f^*|_{\N^1_{\mathbb{C}}(X)}$.
Then $g^*|_{\N^1_{\mathbb{C}}(Y)}=\varphi|_{\N^1_{\mathbb{C}}(Y)}$.
Suppose $\varphi$ is not diagonalizable.
Since $\dim(V/W)=1$,
the Jordan canonical form of $\varphi$ is
$$\begin{pmatrix} \lambda_1 & 1 & 0 &\cdots & 0
\\ 0 & \lambda_2 &0&\cdots & 0\\ 0&0&\lambda_3&\cdots &0\\
\vdots & \vdots &\vdots &\vdots &\vdots\\
0 &0 &0 &\cdots &\lambda_k \end{pmatrix}$$
where $\lambda_1=\lambda_2>0$ by Lemma \ref{lem-eig-int}.
So we may find some $x_1\in V_{\R}\backslash W$ such that $x_2:=\varphi(x_1)-\lambda_1x_1\in W_{\R}$ is a (non-zero) eigenvector of $\lambda_2$. We may further assume $x_1$ is $\pi$-ample.
Let $x_3\in W,\cdots, x_k\in W$ be the eigenvectors of $\lambda_3,\cdots, \lambda_k$, where $k=\dim V$.

We first claim that the intersection number $x_1^{a_1}\cdot x_2^{a_2} \cdots x_k^{a_k}$ is non-zero for $a_1=a$ and suitable $a_2>0, a_3\ge 0\cdots, a_k\ge 0$ such that $\sum\limits_{i=1}^k a_i=m$.
Note that $x_2,\cdots, x_k$ spans $W$. Let $H=\sum\limits_{i\ge 2} b_i x_i$ be an ample divisor class on $Y$.
Since $0\neq x_2\in W_{\R}$, either $x_2\cdot H^{n-1}$ or $x_2^2\cdot H^{n-2}$ is non-zero (cf.~\cite[Lemma 2.3]{MZ}). In particular,  the intersection $x_2^{a_2} \cdots x_k^{a_k}\neq 0$ on $Y$ for some $a_2>0$.
So we may assume $x_2^{a_2} \cdots x_k^{a_k}=cF$ on $X$ for some general fibre $F$ of $\pi$ and non-zero complex number $c$.
Since $x_1$ is $\pi$-ample, $x_1^a\cdot F=(x_1|_F)^a\neq 0$ and hence $x_1^a\cdot x_2^{a_2} \cdots x_k^{a_k}\neq 0$.
So the claim is proved.

We next claim that $\deg f=\prod\limits_{i= 1}^k\lambda_i^{a_i}$ and $\deg g=\prod\limits_{i\ge 2}\lambda_i^{a_i}$.
Applying the projection formula for $g$ on $Y$, we have
$$(\deg g) (x_2^{a_2} \cdots x_k^{a_k})=g^*(x_2^{a_2} \cdots x_k^{a_k})=(\prod\limits_{i\ge 2}\lambda_i^{a_i})(x_2^{a_2} \cdots x_k^{a_k}).$$
Given non-negative integers $s_1,\cdots, s_k$ with $\sum\limits_{i=1}^k s_i=m$ and $s_1<a$, one has $\sum\limits_{i=1}^k s_i>n$ and hence $x_2^{s_2}\cdots x_k^{s_k}=0$. 
Applying the projection formula for $f$ on $X$, we have
$$(\deg f) (x_1^{a_1} \cdots x_k^{a_k})=f^*(x_1^{a_1} \cdots x_k^{a_k})=(\lambda_1x_1+x_2)^{a_1}\cdots (\lambda_kx_k)^{a_k}=(\prod_{i=1}^k\lambda_i^{a_i})(x_1^{a_1} \cdots x_k^{a_k}).$$

Now we have
\begin{align*}
(\prod_{i=1}^k\lambda_i^{a_i})&(x_1^{a_1+1}\cdot x_2^{a_2-1}\cdots x_k^{a_k})
=\deg f (x_1^{a_1+1}\cdot x_2^{a_2-1}\cdots x_k^{a_k})\\
&=(f^*x_1)^{a_1+1}\cdot (f^*x_2)^{a_2-1}\cdots (f^*x_k)^{a_k}\\
&=(\lambda_1x_1+x_2)^{a_1+1}\cdot (\lambda_2x_2)^{a_2-1}\prod_{i\ge 3}(\lambda_ix_i)^{a_k}\\
&=(\lambda_1^{a_1+1}\cdot\lambda_2^{a_2-1}\cdot \prod_{i\ge 3}\lambda_i^{a_i})(x_1^{a_1+1}\cdot x_2^{a_2-1}\cdots x_k^{a_k})+(a_1+1)(\prod_{i= 1}^k\lambda_i^{a_i})(x_1^{a_1}\cdots x_k^{a_k})\\
&=(\prod_{i=1}^k\lambda_i^{a_i})(x_1^{a_1+1}\cdot x_2^{a_2-1}\cdots x_k^{a_k})+(a_1+1)(\prod_{i=1}^k\lambda_i^{a_i})(x_1^{a_1}\cdots x_k^{a_k})
\end{align*}
since $\lambda_1=\lambda_2$.
So $x_1^{a_1}\cdots x_k^{a_k}=0$, a contradiction.
\end{proof}

With all the preparation work settled, we now prove our main theorems.
\begin{proof}[Proof of Theorem \ref{main-thm-mmp}]
 If $K_X$ is pseudo-effective, then (1) follows from Theorem \ref{thm-kx-pe-qa} and (3) is then trivial.
Next, we consider the case where $K_X$ is not pseudo-effective.

By \cite[Corollary 1.3.3]{BCHM}, since $K_X$ is not pseudo-effective, we may run MMP with scaling for a finitely many steps: $X=X_1\dashrightarrow\cdots\dashrightarrow X_j$ (divisorial contractions and flips) and end up with a Mori's fibre space $X_j\to X_{j+1}$.
Note that $X_{j+1}$ is again $\mathbb{Q}$-factorial (cf.~\cite[Corollary 3.18]{KM} with klt singularities (cf.~\cite[Corollary 4.5]{Fu}).
So by running the same program several times, we may get the following sequence:  $$(\ast)\,\,X=X_1\dashrightarrow \cdots \dashrightarrow X_i \dashrightarrow \cdots \dashrightarrow X_r=Y,$$ such that $K_{X_{r}}$ is pseudo-effective.
Replacing $f$ by a positive power, the sequence $(\ast)$ is $f$-equivariant by Theorem \ref{thm-equi-mmp}.
Since $K_{X_{r}}$ is pseudo-effective, $Y=X_r$ is $Q$-abelian by (1).

By Lemma \ref{mor-q-abelian}, the composition $X_i\dashrightarrow Y$ is a morphism for each $i$.
If $X_i\dashrightarrow X_{i+1}$ is a flip, then for the corresponding flipping contraction $X_i\to Z_i$, the pair $(Z_i,\Delta_i)$ is klt for some effective $\mathbb{Q}$-divisor $\Delta_i$ by \cite[Corollary 4.5]{Fu}.
Hence $Z_i\dashrightarrow Y$ is also a morphism by Lemma \ref{mor-q-abelian} again.
Together, the sequence $(\ast)$ is a relative MMP over $Y$.

By \cite[Lemma 2.16]{MZ} and Lemma \ref{fibres-rc+irr}, $X_i\to Y$ is equi-dimensional with every fibre being (irreducible) rationally connected.
Note that $K_{X_i}$ is not pseudo-effective for any $i<r$ by (1).
Then the final map $X_{r-1}\to X_r$ is a Fano contraction.
So (2) is proved.

We show (3) by induction on $i$ from $r$ to $1$. It is trivial when $i=r$.
Suppose $f_{i+1}^*|_{\N^1(X_{i+1})}$ is diagonalizable over $\mathbb{C}$.
Let $\pi:X_i\dasharrow X_{i+1}$ be the $i$-th step of the sequence ($\ast$).
If $\pi$ is a flip, then $\N^1(X_i)=\pi^*\N^1(X_{i+1})$ and hence $f_i^*|_{\N^1(X_i)}$ is diagonalizable over $\mathbb{C}$.
If $\pi$ is a divisorial contraction with $E$ being the $\pi$-exceptional prime divisor, then $f_i^*E=\lambda E$ for some integer $\lambda>1$ by Lemma \ref{lem-eig-int}. Note that $-E$ is $\pi$-ample by \cite[Lemma 2.62]{KM}. Its class $[E]\in \N^1(X_i)\backslash \pi^*\N^1(X_{i+1})$.
Note that $\pi^*\N^1(X_{i+1})$ is a $1$-codimensional subspace of $\N^1(X_i)$. 
Then $f_i^*|_{\N^1(X_i)}$ is diagonalizable over $\mathbb{C}$.
If $\pi:X_i\to X_{i+1}$ is a Fano contraction, then $f_i^*|_{\N^1(X_i)}$ is diagonalizable over $\mathbb{C}$ by Lemma \ref{lem-diag-fano}.
So (3) is proved.
\end{proof}

\begin{proof}[Proof of Theorem \ref{main-thm-rc-diag}]
We apply Theorem \ref{main-thm-mmp} and use the notation there.
Replacing $f$ by a positive power, there is an $f$-equivariant equi-dimensional morphism $\pi:X\to Y$ with all the fibre being irreducible such that $Y$ is $Q$-abelian.

We claim that $Y$ is a point. Suppose $\dim(Y)>0$. Then there is a quasi-\'etale cover $A\to Y$ of degree greater than $1$. Let $X':=X\times_Y A$.
Since $\pi$ is equi-dimensional and has irreducible fibres, then the induced cover $X'\to X$ is quasi-\'etale and hence \'etale of degree greater than $1$ by the purity of branch loci, a contradiction to $X$ being simply connected by \cite[Corollary 4.18]{De}.

Since $Y$ is a point, $f^*|_{\N^1(X)}$ is diagonalizable over $\mathbb{C}$ by Theorem \ref{main-thm-mmp}.
Let $\lambda$ be an eigenvalue of $f^*|_{\N^1(X)}$.
Then $\lambda$ is an eigenvalue of $f_i^*|_{\N^1(X_i)/\pi^*\N^1(X_{i+1})}$ for some $i$, where $\pi:X_i\to X_{i+1}$ is either a divisorial or Fano contraction.
By Lemma \ref{lem-eig-int}, $\lambda>1$ is an integer.
In particular, $f^*|_{\N^1(X)}$ is diagonalizable over $\mathbb{Q}$.
\end{proof}

\section{Some examples}\label{sec-exa}

Let $f:X\to X$ be an int-amplified endomorphism of a projective variety $X$. 
Then $f^*|_{\N^1(X)}$ may not be diagonalizable over $\mathbb{C}$.
\begin{example}[N. Fakhruddin]\label{exa-fak}
\rm{}
Let $X=E\times E$ where $E$ is an elliptic curve admitting a complex multiplication.
Then $\dim (\N^1(X))=4$.
Let $\sigma:X\to X$ be an automorphism via $(x,y)\mapsto (x,x+y)$.
Then $\sigma$ is of null-entropy and $\sigma^*|_{\N^1(X)}$ is not diagonalizable over $\mathbb{C}$.
Let $n_X$ be the multiplication endomorpphism of $X$.
Note that $n_X^*|_{\N^1(X)}=n^2\id_{\N^1(X)}$.
By Theorem \ref{main-thm-cri}, $f:=\sigma\circ n_S$ is int-amplified for $n> 1$.
Clearly, $f^*|_{\N^1(X)}$ is not diagonalizable over $\mathbb{C}$.
\end{example}

Let $f:X\to X$ be an amplified endomorphism of a projective variety $X$.
In general, there do not exist projective varieties $Y$ and $Z$, an int-amplified  endomorphim $g:Y\to Y$, an amplified automorphism  $h:Z\to Z$, and a dominant rational map $\pi:Y\times Z\dasharrow X$ such that $\pi\circ (g\times h)=f\circ\pi$.
\begin{example}\label{exa-amplified}
\rm
Let $X=E\times E$ where $E$ is an elliptic curve.
There is an action of $\SL_2(\mathbb{Z})$ on $X$ by automorphisms.
Take $M\in \SL_2(\mathbb{Z})$ such that some eigenvalue of $M$ is greater than $1$.
Let $f_1:X\to X$ an automorphism determined by $M$.
Then $f_1$ is of positive entropy and we may assume that the spectral radius of $f_1^*|_{\N^1(X)}$ is greater than $4$ after replacing $f_1$ by some positive power.
Let $f=2_X\circ f_1$ where $2_X:X\to X$ is the multiplication endomorphism of $X$.
Note that $(2_X)^*|_{\N^1(X)}=4 \id_{\N^1(X)}$.
So all the eigenvalues of $f^*|_{\N^1(X)}$ are of modulus not equal to $1$.
In particular, $f$ is amplified and not int-amplified by Theorem \ref{main-thm-cri}. 
Suppose the contrary that the above $g$ and $h$ exist. 
By Theorem \ref{thm-fak}, we may assume $g(y)=y$ and $h(z)=z$ for some $y\in Y$ and $z\in Z$ after replacing $g$ and $h$ by some positive power.
In particular, $(g\times h) ( \{y\}\times Z)=\{y\}\times Z$.
Clearly, $\{y\}\times Z$ does not dominate $X$ and $\{y\}\times Z$ is not contracted to a point in $X$ by taking a general $y$. So we may have a curve $C$ in $X$ such that $f(C)=C$ and $f|_C$ is an automorphism.
This is impossible since $f|_C$ is amplified and hence non-isomorphic. 
\end{example}
Let $f:X\to X$ be an int-amplified endomorphism of a projective variety $X$.
In general, there do not exist projective varieties $Y$ and $Z$, polarized endomorphims $g:Y\to Y$, $h:Z\to Z$, and a dominant rational map $\pi:Y\times Z\dasharrow X$ such that $\pi\circ (g\times h)=f\circ\pi$.

\begin{example}\label{exa-intamp}
\rm
Let $X=E\times E$ where $E$ is an elliptic curve admitting a complex multiplication.
Let $f:X\to X$ be an int-amplified endomorphism such that $f(a,b)=(na,na+nb)$ for some integer $n>1$ as constructed in Example \ref{exa-fak}.
Then all the eigenvalues of $f^*|_{\N^1(X)}$ are of modulus $n^2$.
Suppose the contrary that the above $g$ and $h$ exist. 
By a similar argument in Example \ref{exa-amplified}, we have two different curves $E_1$ and $E_2$ in $X$ such that $E_1\cap E_2\neq \emptyset$ and $f^s(E_1)=E_1$, $f^s(E_2)=E_2$ for some $s>0$.
Note that $f^s|_{E_1}$ and $f^s|_{E_2}$ are both amplified and hence polarized.
So $E_1$ and $E_2$ are elliptic curves.
We may assume that $f^s|_{E_1\cap E_2}=\id$.
By choosing an identity element in $E_1\cap E_2$, $E_1$ and $E_2$ can be regarded as subgroups of $X$ and we may assume $f^s$, $f^s|_{E_1}$ and $f^s|_{E_2}$ are isogenies.
Then we have $f^s$-equivariant fibrations $X\to X/E_1$ and $X\to X/E_2$.
So $(f^s)^*E_1\equiv n^{2s}E_1$ and $(f^s)^*E_2\equiv n^{2s}E_2$.
Since $E_1\cdot E_2>0$, $f^s|_{E_1}$ and $f^s|_{E_2}$ are both $n^{2s}$-polarized (cf.~\cite[Introduction]{CMZ}).
Let $\widetilde{f}:=f^s|_{E_1}\times f^s|_{E_2}$.
Then $\widetilde{f}$ is also an $n^{2s}$-polarized isogeny.
Let $\tau: E_1\times E_2\to X$ such that $\tau(a,b)=a+b$.
Then $\tau$ is an isogeny such that $f\circ \tau=\tau\circ \widetilde{f}$.
Therefore $f^s$ is $n^{2s}$-polarized (cf.~\cite[Lemma 3.10 and Theorem 3.11]{MZ}).
However, by Example \ref{exa-fak}, $(f^s)^*|_{\N^1(X)}$ is not diagonalizable over $\mathbb{C}$.
So we get a contradiction by \cite[Proposition 2.9]{MZ}.
\end{example}

We construct two polarized endomorphisms with the same set of preperiodic points such that their composition is not int-amplified and hence not polarized. 
\begin{example}\label{exa-comp}
\rm
Let $X=E\times E$ where $E$ is an elliptic curve admitting a complex multiplication.
Let $f:X\to X$ be a surjective endomorphism corresponding to the matrix $\begin{pmatrix} 1 & -5 \\ 1 & 1 \end{pmatrix}$, i.e., $f(a,b)=(a-5b, a+b)$.
Then $f^*|_{H^{1,0}(X)}$ is diagonalizable with two eigenvalues being of the same modulus $\sqrt{6}$.
Note that $f^*|_{H^{1,1}(X)}=f^*|_{H^{1,0}(X)}\wedge \overline{f^*|_{H^{1,0}(X)}}$ and $\N^1_{\mathbb{C}}(X)=H^{1,1}(X)$.
So $f^*|_{\N^1_{\mathbb{C}}(X)}$ is diagonalizable with four eigenvalues of the same modulus $6$.
Therefore, $f$ is polarized by \cite[Proposition 2.9]{MZ}.
Let $\sigma:X\to X$ be an automorphism corresponding to the matrix $\begin{pmatrix} 1 & -10 \\ 0 & 1 \end{pmatrix}$.
By the same argument, $g:=\sigma^{-1}\circ f\circ \sigma$ is polarized corresponding to the matrix $\begin{pmatrix} 11 & -105 \\ 1 & -9 \end{pmatrix}$.
Denote by $h:=f\circ g$.
Then $h$ corresponds to the matrix $\begin{pmatrix} 6 & -60 \\ 12 & -114 \end{pmatrix}$.
Note that this matrix has a real eigenvalue with modulus less than $1$.
So $h^*|_{\N^1_{\mathbb{C}}(X)}$ has an eigenvalue with modulus less than $1$.
Therefore, $h$ is not int-amplified by Theorem \ref{main-thm-cri}.
Finally, note that both $f$ and $g$ are polarized isogenies.
Then $\Prep(f)=\Prep(g)$ is the set of torsion points of $X$ by \cite[Proposition 2.5]{Kr}.
So the answer to Question \ref{que-yz} is negative.
\end{example}


\begin{thebibliography}{99}
\bibitem{Be}
A. Beauville,
Vari\'et\'e K\"ahlerinnes dont la premiere classe de Chern est nulle,
J. Diff. Geom. \textbf{18}
(1983), 755-782.

\bibitem{BCHM}
C. Birkar, P. Cascini, C. D. Hacon and J. McKernan,
Existence of minimal models for varieties of log general type. J. Amer. Math. Soc., \textbf{23}(2):405-468, 2010.

\bibitem{Bi}
G. Birkhoff, 
Linear transformations with invariant cones, 
Amer. Math. Monthly \textbf{74} (1967), 274-276.

\bibitem{BDPP}
S.~Boucksom, J.-P.~Demailly, M.~Paun and T.~Peternell,
The pseudo-effective cone of a compact K\"ahler manifold and varieties of negative Kodaira dimension,
J. Algebraic Geom. \textbf{22} (2013), no. 2, 201-248.

\bibitem{BFF}
S.~Boucksom, T.~de Fernex and C.~Favre,
The volume of an isolated singularity, Duke Math. J. \textbf{161} (2012), no. 8, 1455-1520.

\bibitem{BH}
A.~Broustet and A.~H\"oring,
Singularities of varieties admitting an endomorphism,
Math. Ann. \textbf{360} (2014), no. 1-2, 439-456.

\bibitem{CMZ}
P.~Cascini, S.~Meng and D.-Q.~Zhang,
Polarized endomorphisms of normal projective threefolds in arbitrary characteristic,
\href{http://arXiv.org/abs/1710.01903}{ arxiv:1710.01903}

\bibitem{De}
O.~Debarre,
Higher-dimensional algebraic geometry, Universitext. Springer-Verlag, New York, 2001.

\bibitem{Fak}
N.~Fakhruddin,
Questions on self-maps of algebraic varieties,
J. Ramanujan Math. Soc., \textbf{18}(2):109-122, 2003.

\bibitem{FGI}
B.~Fantechi, L.~G\"ottsche, L.~Illusie, S.~L.~Kleiman, N.~Nitsure and A.~Vistoli,
Fundamental algebraic geometry,
Grothendieck's FGA explained, Mathematical Surveys and Monographs, \textbf{123}, American Mathematical Society, Providence, RI, 2005.


\bibitem{Fu}
O.~Fujino, Applications of Kawamata's positivity theorem, Proc. Japan
Acad. Ser. A Math. Sci. \textbf{75} (1999), no. 6, 75-79.

\bibitem{Fu11}
O.~Fujino, Fundamental theorems for the log minimal model program,
Publ. Res. Inst. Math. Sci. \textbf{47} (2011), no. 3, 727-789.

\bibitem{FKL}
M.~Fulger, J.~Koll\'ar, B.~Lehmann,
Volume and Hilbert function of $\mathbb{R}$-divisors,
Michigan Math. J. \textbf{65} (2016), no. 2, 371-387.

\bibitem{GKP}
D.~Greb, S.~Kebekus and T.~Peternell,
\'Etale fundamental groups of Kawamata log terminal spaces, flat sheaves, and quotients of Abelian varieties,
Duke Math. J. \textbf{165} (2016), no. 10, 1965-2004.


\bibitem{HMZ}
F.~Hu, S.~Meng and D.-Q.~Zhang,
Ampleness of canonical divisors of hyperbolic normal projective varieties,
Math. Z. \textbf{278} (2014), no. 3-4, 1179-1193.

\bibitem{KM}
J.~Koll\'ar and S.~Mori,
Birational geometry of algebraic varieties,
Cambridge Tracts in Math.,
\textbf{134} Cambridge Univ. Press, 1998.

\bibitem{Kr}
H.~Krieger and P.~Reschke,
Cohomological conditions on endomorphisms of projective varieties,
Bull. Soc. Math. France \textbf{145} (2017), no. 3, 449-468.


\bibitem{Lang}
S.~Lang,
Abelian varieties, Springer-Verlag, New York-Berlin, 1983.

\bibitem{MZ}
S.~Meng and D. -Q.~Zhang,
Building blocks of polarized endomorphisms of normal projective varieties,
Adv. Math. \textbf{325} (2018), 243-273.

\bibitem{Mu}
D.~Mumford,
Abelian varieties, With appendices by C. P. Ramanujam and Yuri Manin, Corrected reprint of the second (1974) edition, Tata Institute of Fundamental Research Studies in Mathematics, \textbf{5}, Published for the Tata Institute of Fundamental Research, Bombay.

\bibitem{ENS}
N.~Nakayama, On complex normal projective surfaces admitting non-isomorphic surjective endomorphisms,
Preprint 2 September 2008.

\bibitem{ZDA} N.~Nakayama,
Zariski-decomposition and abundance,
MSJ Memoirs Vol.~\textbf{14}, Math.\ Soc.\ Japan, 2004.

\bibitem{Na10}
N.~Nakayama,
Intersection sheaves over normal schemes,
J. Math. Soc. Japan \textbf{62} (2010), no. 2, 487-595.

\bibitem{NZ09}
N.~Nakayama and D.-Q.~Zhang,
Building blocks of \'etale endomorphisms of complex projective manifolds,
Proc. Lond. Math. Soc. (3) \textbf{99} (2009), no. 3, 725-756. 

\bibitem{Na-Zh}
N.~Nakayama and D.-Q.~Zhang,
Polarized endomorphisms of complex normal varieties,
Math. Ann. \textbf{346} (2010), no. 4, 991-1018.

\bibitem{Re}
P.~Reschke,
Distinguished line bundles for complex surface automorphisms,
Transform. Groups, \textbf{19}(1):225-246, 2014.

\bibitem{Ue}
K.~Ueno,
Classification of algebraic varieties, I,
Compos. Math., \textbf{27}(3):277-342, 1973.

\bibitem{Wa}
J.~Wahl,
A characteristic number for links of surface singularities,
J. Amer. Math. Soc. \textbf{3} (1990), no. 3, 625-637.

\bibitem{Ya}
S.~T.~Yau, On the Ricci curvature of a compact K\"ahler manifold and the complex Monge-Amp\`ere
equations, I, Comm. Pure and Appl. Math. \textbf{31} (1978), 339-411.

\bibitem{YZ}
X. Yuan and S. Zhang,
The arithmetic Hodge index theorem for adelic line bundles, 
Math. Ann. \textbf{367} (2017), no. 3-4, 1123-1171. 

\bibitem{Zh-comp}
D.-Q.~Zhang, Polarized endomorphisms of uniruled varieties. Compos. Math. \textbf{146} (2010), no.
1, 145-168.

\bibitem{Zh-tams}
D.-Q.~Zhang, $n$-dimensional projective varieties with the action of an abelian group of rank $n - 1$,
Trans. Amer. Math. Soc. \textbf{368} (2016), no. 12, 8849-8872.

\bibitem{Zhsw}
S.~W.~Zhang, “Distributions in algebraic dynamics, 381-430,” Survey in Differential
Geometry \textbf{10}, Somerville, MA, International Press, 2006.


%
%
%
%

\end{thebibliography}
\end{document}